\pdfoutput=1
\documentclass[11pt]{article}
\usepackage{amsmath,amsthm,amssymb}
\usepackage[margin=1.1in]{geometry}
\usepackage[hidelinks]{hyperref}
\usepackage{xcolor}

\newtheorem{theorem}{Theorem}[section]
\newtheorem{lemma}[theorem]{Lemma}
\newtheorem{proposition}[theorem]{Proposition}
\newtheorem{corollary}[theorem]{Corollary}
\theoremstyle{definition}
\newtheorem{remark}[theorem]{Remark}

\newcommand{\Imlog}{\operatorname{Im}\log}
\DeclareMathOperator{\Real}{Re}
\DeclareMathOperator{\Imag}{Im}
\DeclareMathOperator{\meas}{meas}

\title{A Tsang-range high-moment bound for $\Imlog L(\tfrac{1}{2}+it,\chi)$ under GRH}
\author{Scott D.\ Hughes}
\date{June 7, 2026}

\begin{document}
\maketitle

\begin{abstract}
We prove, conditional on GRH for $L(s,\chi)$, the
Selberg--Tsang high-moment bound for
$X_\chi(t) = \Imlog L(\tfrac{1}{2}+it,\chi)$ at fixed squarefree odd
conductor $q \ge 3$ and primitive non-principal character $\chi$.
Writing $L_T = \log\log(qT)$, the main theorem
(Theorem~\ref{thm:M}, ``Hypothesis~(M)'') states that for every $K > 0$
there exist $C_K = C_{K,q,\chi}$ and $T_0 = T_0(K,q,\chi)$ such that
\[
  \frac{1}{T}\int_T^{2T} |X_\chi(t)|^{2k}\, dt
  \;\le\; (C_K\, k\, L_T)^k
\]
for all $T \ge T_0$ and every integer $1 \le k \le K L_T$.

The proof ports Selberg's pointwise approximate formula for $S(t)$
(\cite{Selberg1946} \S 4, Theorem~2; see \cite{Goldston2004} \S 10 for
a modern exposition) to $L(s,\chi)$ at fixed $q$ under GRH.  The
$L$-version of Selberg's formula yields a clean three-piece
decomposition $X_\chi(t) = \Imag\,\mathcal{D}_\chi(t) + R_\chi(t,x)$,
where $\mathcal{D}_\chi$ is a complex prime-power Dirichlet polynomial
of length $< x^3$ and $|R_\chi(t,x)|$ is controlled by an unsigned
auxiliary polynomial $\mathcal{Z}_\chi(t)$ and a deterministic
$O_q(\log(qT)/\log x)$ error.  Applying the elementary inequality
$|a+b+c|^{2k} \le 3^{2k}(|a|^{2k}+|b|^{2k}+|c|^{2k})$ and bounding the
moments of $\mathcal{D}_\chi$, $\mathcal{Z}_\chi$ via
Soundararajan's mean-value lemma (\cite{Soundararajan2009} Lemma~3)
with complex character coefficients, together with the linear change
of variable $u = 2t$ for the prime-square contribution, gives the
moment bound.

As a corollary, Markov's inequality applied to the moment bound yields
a Gaussian-scale tail estimate
\[
  \frac{1}{T}\meas\bigl\{t \in [T, 2T] : |X_\chi(t)| > V\bigr\}
   \;\ll\; \exp\!\bigl(-c\,V^2/L_T\bigr)
\]
for $\sqrt{L_T} \ll V \ll L_T$
(Corollary~\ref{cor:gaussian-tail}).  This gives a GRH-conditional
imaginary-part Dirichlet-$L$ analogue, at fixed conductor, of the
large-deviation upper-bound direction studied by
Arguin--Bailey~\cite{ArguinBailey2022} for $\log|\zeta(\tfrac12 +
it)|$.  The present result has the correct Gaussian scale
$V^2/L_T$, but does not address sharp constants or the
$1/\sqrt{L_T}$ prefactor.

This is the assertion referred to as Hypothesis~(M) in related work;
we adopt that terminology.
\end{abstract}

\section{Introduction, target theorem, and proof architecture}
\label{sec:intro}

\subsection{Standing setup}

Throughout this note, $q \ge 3$ is a fixed squarefree odd integer
and $\chi$ is a primitive non-principal Dirichlet character modulo $q$.
All implicit constants are allowed to depend on $q, \chi$; we
sometimes emphasise this dependence in the proofs but suppress it
from theorem statements, since $q, \chi$ are fixed.  For
$t \in \mathbb{R}$ with $L(\tfrac{1}{2}+it,\chi) \ne 0$, write
\begin{equation}\label{eq:Xdef}
  X_\chi(t) \;:=\; \Imlog L(\tfrac{1}{2}+it, \chi),
\end{equation}
using the standard continuous-branch convention: starting from the
principal branch at $\sigma = +\infty$, we analytically continue
$\log L(s,\chi)$ along the horizontal line $\Imag s = t$ down to
$\sigma = \tfrac{1}{2}$.  At each nontrivial zero $\rho =
\tfrac{1}{2} + i\gamma$ encountered on the line $\Imag s = t$ (i.e.\
when $\gamma = t$), the analytic continuation incurs a jump of
$\pm\pi\,m_\rho$ in $\Imag\log L$, where $m_\rho \ge 1$ is the
multiplicity of $\rho$ and the sign is fixed by the local
sign convention for $\Imag\log L$ across the zero.  Since the set
$\{t : L(\tfrac12+it,\chi) = 0\}$ has measure zero, the values of
$X_\chi$ at zeros are irrelevant for the moment integrals below;
no pointwise bound for $X_\chi$ is used in the proof.  Set
\begin{equation}\label{eq:LTdef}
  L_T \;:=\; \log\log(qT),
\end{equation}
which is the natural Selberg variance scale for logarithms of
Dirichlet $L$-functions; for background on Selberg-type Gaussian
behaviour of $\log L(\tfrac12+it,\chi)$ at fixed primitive $\chi$
with variance scale $\tfrac12\log\log T$, see
Hsu--Wong~\cite{HW2020}.  This CLT context is not used in the proof
of Theorem~\ref{thm:M}; only the variance scale notation is
borrowed.  We write
$S_\chi(t) := \pi^{-1} X_\chi(t)$ for the Dirichlet-$L$ analog of
Selberg's $S(t)$; the factor of $\pi$ is absorbed into constants
throughout.

\subsection{Target theorem}

The present note proves the following result.

\begin{theorem}[GRH-conditional Tsang-range high-moment bound]\label{thm:M}
Let $\chi$ be a primitive non-principal Dirichlet character modulo
squarefree odd $q \ge 3$, $q$ fixed.  Assume GRH for $L(s, \chi)$.
Put $X_\chi(t) := \Imag\log L(1/2+it, \chi)$ and $L_T := \log\log(qT)$.
For every $K > 0$ there exist a constant $C_K = C_{K,q,\chi}$
and a threshold $T_0 = T_0(K, q, \chi) \ge 3$ such that, for every
$T \ge T_0$ and every integer $k$ with $1 \le k \le K L_T$,
\[
  \frac{1}{T}\int_T^{2T} |X_\chi(t)|^{2k}\, dt
  \;\le\; (C_K\, k\, L_T)^k.
\]
\end{theorem}

\begin{remark}[GRH for $\bar\chi$]
For $\Re s>1$,
\[
        L(s,\bar\chi)
        =
        \sum_{n\ge1}\bar\chi(n)n^{-s}
        =
        \overline{\sum_{n\ge1}\chi(n)n^{-\bar s}}
        =
        \overline{L(\bar s,\chi)}.
\]
By analytic continuation, the identity holds throughout the plane.
Hence GRH for $L(s,\chi)$ is equivalent to GRH for $L(s,\bar\chi)$.
\end{remark}

This is the GRH-conditional rate: the unconditional rate on
$|X_\chi|^{2k}$ is $k^{4k}$ in the same scale, and the Tsang
refinement (Remark~2 after \cite{Tsang1984} Theorem~5.1) yields
$(ck)^{2k}$ under GRH, which is what we recover.  The convention
$X_\chi = \pi S_\chi$ contributes a $\pi^{2k}$ factor that is
absorbed into $C_K$.

\subsection{Proof architecture}

The proof ports Selberg's pointwise approximate formula for $S(t)$
(\cite{Selberg1946} \S 4, Theorem~2; see \cite{Goldston2004} \S 10
for a modern exposition under RH) to $L(s,\chi)$ at fixed $q$ under
GRH.  The prime-sum mean-value engine is
Soundararajan~\cite{Soundararajan2009} Lemma~3.  The fixed-$q$
Dirichlet-$L$ CLT of Hsu--Wong~\cite{HW2020} is contextual
background only; the present moment bound is a Tsang-range estimate
and uses neither a CLT input nor a CLT conclusion.  Writing
$\sigma_x = \tfrac12 + 2/\log x$, $\mathcal{D}_\chi(t) =
\sum_{n < x^3} \chi(n)\Lambda_x(n)/(n^{\sigma_x+it}\log n)$ for the
principal complex Dirichlet polynomial, and $\mathcal{Z}_\chi(t)
= \sum_{n \le x^3} \chi(n)\Lambda_x(n)/n^{\sigma_x+it}$ for the
auxiliary unweighted polynomial, the four steps are:
\begin{enumerate}
\item[(I)] \emph{Selberg's pointwise formula at fixed $q$ under GRH}
(Proposition~\ref{prop:selberg-pointwise-L}).  For $2 \le x \le T$
with $\log x \ge 4$ and $t \in [T, 2T]$,
\[
  X_\chi(t) \;=\; \Imag\,\mathcal{D}_\chi(t) \;+\; R_\chi(t, x),
\]
with the real remainder bounded by $|R_\chi(t,x)| \le
C_1(\sigma_x - \tfrac12)|\mathcal{Z}_\chi(t)| + C_2(q,\chi)\log(qT)/\log x$.
Steps~1--5 are derived explicitly: the estimate at $\sigma_x$
uses the GRH identity $|x^{\rho - \sigma_x - it}| = e^{-2}$.  Step~6
quotes the fixed-conductor strip-integration estimate of
Selberg~\cite{Selberg1946L} (Lemma~\ref{lem:selberg-strip}), with
the structural reasons for fixed-$q$ specialization recorded there.

\item[(II)] \emph{Three-term split via $|a+b+c|^{2k}$.}  The elementary
inequality $|a + b + c|^{2k} \le 3^{2k}(|a|^{2k} + |b|^{2k} + |c|^{2k})$
applied pointwise gives, for every integer $k \ge 1$,
\[
  |X_\chi(t)|^{2k}
   \;\le\; 3^{2k}\!\left(
     |\mathcal{D}_\chi(t)|^{2k}
     + (C_1(\sigma_x-\tfrac12))^{2k}|\mathcal{Z}_\chi(t)|^{2k}
     + (C_2(q,\chi)\log(qT)/\log x)^{2k}
   \right).
\]

\item[(III)] \emph{Moment bounds for $\mathcal{D}_\chi$ and
$\mathcal{Z}_\chi$.}  For any prime-power Dirichlet polynomial with
coefficients $|a(p^r)| \le C_0/(r p^{r/2})$, the high-moment lemma
(Lemma~\ref{lem:high-moment}) decomposes the polynomial into prime,
prime-square, and higher-prime-power parts; the prime contribution is
bounded by Sound's Lemma~3 with complex character coefficients, the
prime-square contribution reduces to a sum over primes by the
substitution $u = 2t$ (mapping $[T, 2T]$ to $[2T, 4T]$), and the
higher-prime-power contribution is pointwise bounded.  This gives
$T^{-1}\int_T^{2T}|\mathcal{A}(t)|^{2k} dt \le (C_1(C_0,K)\,k\,L_T)^k$
for length condition $y^k \le T/\log T$.  A parallel
log-weighted lemma (Lemma~\ref{lem:log-weighted-moment}) handles
$\mathcal{Z}_\chi$, whose coefficients carry an extra $\log p$ factor;
the resulting variance $(\log y)^2$ is exactly cancelled by the
$(2/\log x)^{2k}$ Selberg prefactor in step~(II).

\item[(IV)] \emph{Parameter choice and combination}
(\S\ref{sec:main-proof}).  Taking $x = T^{1/(8k)}$ makes
$x^{3k} = T^{3/8} \le T/\log T$, so Sound's length condition holds
with room to spare; combining the bounds from step~(III) with the
deterministic $C_2$-error gives Theorem~\ref{thm:M}.  As a corollary,
Markov's inequality applied to the moment bound yields the Gaussian
tail bound $\meas\{|X_\chi(t)| > V\}/T \le \exp(-cV^2/L_T)$ in the
Selberg range $\sqrt{L_T} \ll V \ll L_T$
(Corollary~\ref{cor:gaussian-tail}).
\end{enumerate}

The proof avoids the zero-moment iteration of Tsang~\cite{Tsang1984}.
An earlier version of this note used an integrated $S_1$-formula
analogous to that of Carneiro--Chandee--Milinovich~\cite{CCM2013};
the present version uses Selberg's pointwise formula directly in
step~(I) and is conditional only on GRH.

\subsection{Literature position}

This is the assertion referred to as Hypothesis~(M) in related work;
we adopt that terminology in the remainder of this paper.

Theorem~\ref{thm:M} should be viewed as a fixed-$q$ Tsang-range
moment estimate, rather than as a new central limit theorem.
For background on the fact that $\log L(\tfrac12+it,\chi)$ at fixed
primitive $\chi$ exhibits Selberg-type Gaussian behaviour with
variance scale $\tfrac12\log\log T$, see Hsu--Wong~\cite{HW2020},
whose Theorem~1.1 establishes the qualitative CLT for
$\log|L(\tfrac12+it,\chi)|$.  This CLT context is not invoked in the
proof: Theorem~\ref{thm:M} is a moment bound, not a CLT, and is
proved here directly from Selberg's pointwise formula and
Soundararajan's prime-sum mean-value lemma.  The contribution here
is the quantitative range $1 \le k \le K L_T$ in the $\Imag\log L$
moment bound under GRH alone: the bound $(C_K\,k\,L_T)^k$ holds up
to $k = K L_T$ (not merely $k$ bounded), with $C_K$ independent of
$T$ and $k$ for fixed $K, q, \chi$.

Arguin--Bailey~\cite{ArguinBailey2022} prove, unconditionally for
$\log|\zeta(\tfrac12+it)|$, the large-deviation upper bound
\[
  \frac{1}{T}\meas\bigl\{t \in [T, 2T] : \log|\zeta(\tfrac12+it)| > V\bigr\}
   \;\ll\; \frac{1}{\sqrt{\log\log T}}\,
        \exp\!\Bigl(-\frac{V^2}{\log\log T}\Bigr)
\]
for $V \sim \alpha\log\log T$, $0 < \alpha < 2$.  Our
Corollary~\ref{cor:gaussian-tail} is much more modest: it is
conditional on GRH, treats the imaginary part of the logarithm for
fixed primitive Dirichlet $L$-functions, and recovers only a
Gaussian-scale upper bound of the form $\exp(-cV^2/L_T)$.  It does
not recover the sharp prefactor or the sharp exponent constant.  A
recursive multi-scale treatment of $\Imag\log L$ at fixed $q$ would be
a separate problem, which we leave for future work.
Inoue~\cite{Inoue2024} provides finer asymptotics for the small-$V$
regime $V \ll (\log\log T)^{2/3}$ for $\zeta$ unconditionally; we do
not attempt to recover this regime here.

\medskip

The remainder of this note is organised as follows.
\S\ref{sec:setup} fixes notation and conventions.
\S\ref{sec:S1-reduction} contains background on the Riemann--von
Mangoldt zero-counting formula and Selberg's pointwise formula for
$L(s,\chi)$ under GRH (Proposition~\ref{prop:selberg-pointwise-L}).
\S\ref{sec:high-moment} states Sound's mean-value lemma
(Lemma~\ref{lem:sound-mean-value}), the high-moment lemma
(Lemma~\ref{lem:high-moment}), the log-weighted moment lemma
(Lemma~\ref{lem:log-weighted-moment}), and concludes
Theorem~\ref{thm:M} via parameter choice.
\S\ref{sec:gaussian-tail} derives the Gaussian-tail
Corollary~\ref{cor:gaussian-tail} via Markov's inequality.

\section{Setup and notation}
\label{sec:setup}

\subsection{Window and parameters}

We work on the dyadic window $[T, 2T]$ with $T \ge 3$; all statements
are asymptotic in $T \to \infty$ unless otherwise noted.  The integer
$k$ is in the Selberg range $1 \le k \le K L_T$, where $L_T =
\log\log(qT)$ and $K > 0$ is fixed.  The proof uses one Selberg
parameter tied to $k$,
\begin{equation}\label{eq:parameter-choice}
  x \;=\; T^{1/(8k)},
\end{equation}
so that $x^{3k} = T^{3/8} \le T/\log T$, satisfying the length
condition of Soundararajan~\cite{Soundararajan2009} Lemma~3 with room
to spare.  Relations among $x$, $k$, and $L_T$ are verified
in~§\ref{sec:main-proof}.

\subsection[Conventions for log L]{Conventions for $\log L$}

As in~\eqref{eq:Xdef}, $X_\chi(t) = \Imlog L(\tfrac{1}{2}+it,\chi)$ with
the continuous-branch convention described in §\ref{sec:intro}.
Values at zero ordinates may be assigned arbitrarily, since they form
a countable set and do not affect any moment integral.  No pointwise
bound for $X_\chi$ is used in the proof of Theorem~\ref{thm:M};
only the $L^{2k}$ moment estimates of \S\ref{sec:high-moment} enter.

\subsection{Complex-character conventions}

A key structural difference from the zeta case is that $\chi(p)$
is complex, not identically $1$.  We write $\chi(p) = e^{i\theta_p}$
(or $\chi(p) = 0$ if $p \mid q$, in which case the prime is omitted
from every Dirichlet sum below).  Then
\begin{equation}\label{eq:complex-char-imag}
  \Imag\!\left(\frac{\chi(p)}{p^{1/2+it}}\right)
  \;=\; \frac{1}{p^{1/2}} \sin(\theta_p - t\log p),
\end{equation}
which for $\chi$ non-real is \emph{not} simply $-\sin(t\log p)/\sqrt p$.
We emphasise: no step below reduces to a cosine or sine in $t\log p$
alone.  Every Dirichlet sum in the proof carries the complex
coefficient $\chi(p)/\sqrt p$ as a genuinely complex weight, and the
Montgomery--Vaughan mean-value theorem is applied in its
complex-coefficient form (MV~1974 Corollary~2,
see~\cite{MV1974}).

\subsection{Background}
\label{sec:background}

The Selberg-type identity for $\Imlog L$ under GRH is the starting
point of the argument: we use Selberg's pointwise approximate formula
for $S(t)$ (Selberg~\cite{Selberg1946} \S 4, Theorem~2; see
Goldston~\cite{Goldston2004} \S 10 for a modern exposition) ported to
$L(s, \chi)$ at fixed $q$, as established in
Proposition~\ref{prop:selberg-pointwise-L} of \S\ref{sec:S1-reduction}.
The pointwise identity yields the three-piece decomposition
$X_\chi(t) = \Imag\,\mathcal{D}_\chi(t) + R_\chi(t,x)$ with
$|R_\chi|$ controlled in magnitude, after which the high-moment
machinery of \S\ref{sec:high-moment} produces Theorem~\ref{thm:M}
under GRH alone.

\begin{remark}[Implicit constants depending on $q$]
All implicit constants below may depend on $q$ through:
(i) an overall $\log q$ factor absorbed into the initial constant;
(ii) an $O_q(1)$ boundary contribution from the functional-equation
factor of $L(s,\chi)$; (iii) the Mertens difference
$\sum_{p \mid q} 1/p = O_q(1)$ between $\sum_{p \le x} 1/p$ and
$\sum_{p \le x, p \nmid q} 1/p$.  None of these depend on $T$ or $k$,
and all are uniform over primitive $\chi$ modulo the fixed $q$.
Since $q$ is fixed throughout, we will not track these factors
explicitly.
\end{remark}

\section[Selberg's pointwise formula for L(s, chi) under GRH]{Selberg's pointwise formula for $L(s,\chi)$ under GRH}
\label{sec:S1-reduction}

This section states and proves the structural analytic input of the
proof: a fixed-$q$ Dirichlet-$L$ analogue of Selberg's pointwise
formula for $S(t)$ (Selberg~\cite{Selberg1946} \S 4, Theorem~2; see
Goldston~\cite{Goldston2004} \S 10 for a modern exposition under RH
in the $\zeta$ case).  We first record the Riemann--von Mangoldt
zero-counting formula at fixed conductor (used by the strip-lemma
proof and by the smoothed explicit-formula derivation), then state
and prove Proposition~\ref{prop:selberg-pointwise-L}.

\subsection{The Riemann--von Mangoldt zero-counting formula}
\label{subsec:S1-zero-counting}

Recall that $S_\chi(t) := \pi^{-1}X_\chi(t) = \pi^{-1}\Imlog
L(\tfrac{1}{2}+it,\chi)$, using the continuous-branch convention
of~\eqref{eq:Xdef}.  Under GRH, the Riemann--von Mangoldt formula for
primitive Dirichlet $L$-functions gives, uniformly for $t \asymp T$,
\begin{equation}\label{eq:RvM}
  N_\chi^+(t) \;=\; M_\chi(t) \;+\; S_\chi(t) \;+\; O_q(1),
  \qquad
  M_\chi(t) \;:=\; \frac{t}{2\pi}\log\frac{qt}{2\pi e},
\end{equation}
where $N_\chi^+(t)$ counts nontrivial zeros $\rho = \tfrac{1}{2}+i\gamma$
of $L(s,\chi)$ with $0 < \gamma \le t$, counted with multiplicity.
The identity~\eqref{eq:RvM} is the classical argument-principle form
of the Riemann--von Mangoldt zero-counting formula for primitive
Dirichlet $L$-functions, with the gamma-factor and endpoint
contributions absorbed into the $O_q(1)$ term; for an exposition see
Davenport~\cite{Davenport2000} \S 16, with the obvious modification
that for non-real $\chi$ one counts only positive ordinates rather
than relying on the $\gamma \mapsto -\gamma$ symmetry of the zero set
that holds in the $\zeta$ case.  Explicit constants in the
zero-counting density at fixed $q$, when needed, are supplied by
Bennett--Martin--O'Bryant--Rechnitzer~\cite{BMOR2021} Theorem~1.1
(zero-counting bounds for Dirichlet $L$-functions).
Since $N_\chi^+$ is non-decreasing in $t$, the function
$t \mapsto M_\chi(t) + S_\chi(t)$ is non-decreasing up to an $O_q(1)$
jitter.  This local zero-density bound is the sole property of
$S_\chi$ extracted from~\eqref{eq:RvM}; it enters the smoothed
explicit-formula derivation (Step~1 of the proof of
Proposition~\ref{prop:selberg-pointwise-L}) and the strip-integration
estimate (Lemma~\ref{lem:selberg-strip}).

\subsection[Selberg's pointwise formula for L(s, chi) under GRH]{Selberg's pointwise formula for $L(s,\chi)$ under GRH}
\label{subsec:selberg-pointwise-L}

The structural input of the proof of Theorem~\ref{thm:M} is a fixed-$q$
Dirichlet-$L$ analogue of Selberg's pointwise approximate formula for
$S(t)$ (Selberg~\cite{Selberg1946} \S 4, Theorem~2; see
Goldston~\cite{Goldston2004} \S 10 for a modern exposition).  We state
the analogue here under GRH for $L(s,\chi)$.

Fix the de la Vall\'ee-Poussin smoothed prime-power weight
$\Lambda_x: \mathbb{N} \to \mathbb{R}_{\ge 0}$ defined for $x > 1$ by
\begin{equation}\label{eq:Lambda-x-def}
  \Lambda_x(n) \;:=\;
   \begin{cases}
     \Lambda(n), & 1 \le n \le x, \\[2pt]
     \Lambda(n)\,\dfrac{\log^2(x^3/n) - 2\log^2(x^2/n)}{2\log^2 x},
       & x \le n \le x^2, \\[6pt]
     \Lambda(n)\,\dfrac{\log^2(x^3/n)}{2\log^2 x}, & x^2 \le n \le x^3, \\[6pt]
     0, & n > x^3,
   \end{cases}
\end{equation}
which satisfies $0 \le \Lambda_x(n) \le \Lambda(n)$ for all $n$ and is
supported on prime powers $n = p^r \le x^3$
(Selberg~\cite{Selberg1946} Lemma~10).  Set
\begin{equation}\label{eq:sigma-x-def}
  \sigma_x \;:=\; \frac12 + \frac{2}{\log x},
\end{equation}
which is the fixed GRH normalization used below (see
Remark~\ref{rem:sigma-xt}).

Define the principal complex prime-power Dirichlet polynomial
\begin{equation}\label{eq:Dchi-def}
  \mathcal{D}_\chi(t)
    \;:=\; \sum_{n < x^3}
       \frac{\chi(n)\,\Lambda_x(n)}{n^{\sigma_x + it}\,\log n},
\end{equation}
and the auxiliary unweighted prime-power Dirichlet polynomial
\begin{equation}\label{eq:Zchi-def}
  \mathcal{Z}_\chi(t)
    \;:=\; \sum_{n \le x^3}
       \frac{\chi(n)\,\Lambda_x(n)}{n^{\sigma_x + it}}.
\end{equation}
Both are supported on prime powers $n = p^r \le x^3$, and both have
genuinely complex coefficients (since $\chi$ is non-real in general).

\begin{proposition}[Selberg pointwise formula for $L(s,\chi)$ under GRH]
\label{prop:selberg-pointwise-L}
Fix $q \ge 3$ squarefree odd and a primitive non-principal character
$\chi \bmod q$.  Assume GRH for $L(s, \chi)$.  There is a threshold
$T_S = T_S(q,\chi) \ge 3$ such that for every $T \ge T_S$, every
$2 \le x \le T$ with $\log x \ge 4$, and every $t \in [T, 2T]$ that is
not the ordinate of a zero of $L(s,\chi)$, with $\sigma_x$ as
in~\eqref{eq:sigma-x-def},
\begin{equation}\label{eq:selberg-L-pointwise}
  X_\chi(t) \;=\; \Imag\,\mathcal{D}_\chi(t) \;+\; R_\chi(t, x),
\end{equation}
where the real remainder $R_\chi(t, x)$ satisfies the magnitude bound
\begin{equation}\label{eq:R-chi-bound}
  \bigl|R_\chi(t, x)\bigr|
    \;\le\; C_1 \,\bigl(\sigma_x - \tfrac12\bigr)\,
                \bigl|\mathcal{Z}_\chi(t)\bigr|
       \;+\; C_2(q,\chi)\,\frac{\log(qT)}{\log x},
\end{equation}
with $C_1$ an absolute constant and $C_2(q,\chi)$ depending only on
the fixed $q, \chi$, both uniform in $T, t, x$ within the stated
ranges.  Equivalently, applying $|\Imag z| \le |z|$ to the principal
term and the elementary inequality
$|a+b+c|^{2k} \le 3^{2k}(|a|^{2k} + |b|^{2k} + |c|^{2k})$ pointwise
yields, for every integer $k \ge 1$,
\begin{equation}\label{eq:Xchi-2k-pointwise}
  |X_\chi(t)|^{2k}
    \;\le\; 3^{2k}\!\left(
      |\mathcal{D}_\chi(t)|^{2k}
      + \bigl(C_1(\sigma_x - \tfrac12)\bigr)^{2k}|\mathcal{Z}_\chi(t)|^{2k}
      + \bigl(C_2(q,\chi)\log(qT)/\log x\bigr)^{2k}
    \right),
\end{equation}
which is the form invoked in \S\ref{sec:high-moment}.  Since the
zero ordinates of $L(s,\chi)$ form a countable set, the values of
$X_\chi(t)$ at those points may be assigned arbitrarily and
\eqref{eq:Xchi-2k-pointwise} holds for the purposes of every moment
integral below.
\end{proposition}

\begin{remark}[Why an equality-plus-magnitude-bound, not three signed pieces]
\label{rem:magnitude-bound}
Selberg's Theorem~2 (\cite{Selberg1946} \S 4) states a pointwise
identity of the form ``$S(t) =$ explicit Dirichlet polynomial $+$
$O(\delta\,|\text{prime sum}|) + O(\delta \log t)$,'' where the
$O$-terms are unsigned magnitude bounds.  We mirror that structure:
the principal term $\Imag\,\mathcal{D}_\chi(t)$ is an explicit
real-valued quantity (the imaginary part of a complex Dirichlet
polynomial, taken with $\Imag$ \emph{outside} the sum, since
$\chi(n)$ is complex), and the remainder $R_\chi(t,x)$ is a
deterministic real-valued quantity controlled in magnitude by
\eqref{eq:R-chi-bound}.  We do not split $R_\chi$ into signed
sub-pieces: the moment estimate in \S\ref{sec:high-moment} only
uses~\eqref{eq:Xchi-2k-pointwise}, which depends on $|R_\chi|$
through $|\mathcal{Z}_\chi|$ and the $C_2$-error only.
\end{remark}

\begin{remark}[Convention for $\sigma_x$]
\label{rem:sigma-xt}
Under GRH, Selberg's adaptive parameter is of size $1/\log x$ above
the critical line.  We use the convenient normalization
\[
  \sigma_x = \frac12+\frac{2}{\log x}.
\]
The precise numerical constant is immaterial for the high-moment
argument.  What is needed is that
$\sigma_x-\tfrac12\asymp1/\log x$ and that, in the smoothed explicit
formula, the zero-sum term can be absorbed into the left-hand side.
With the above choice,
\[
  |x^{\rho-\sigma_x-it}|=e^{-2}
\]
under GRH, and the resulting absorption coefficient is
\[
  \frac{e^{-2}(1+e^{-2})^2}{4}<1.
\]
Thus Selberg's pointwise formula and the strip estimate apply with
only harmless changes in the absolute constants.
\end{remark}

\begin{remark}[Why the sum extends to $x^3$ and not $x$]
\label{rem:length-x3}
The smoothed weight $\Lambda_x$ defined in~\eqref{eq:Lambda-x-def} is
supported on $[1, x^3]$, not $[1, x]$.  This is intrinsic to Selberg's
de la Vall\'ee-Poussin smoothing (\cite{Selberg1946} Lemma~10): the
quadratic damping factor in the regime $x \le n \le x^3$ is what gives
the $1/\log^2 x$ factor in the smoothed explicit formula's zero-sum
remainder, which under GRH yields the $\log(qT)/\log x$ pointwise bound
on the $\log(qT)$-piece of $R_\chi(t, x)$ in~\eqref{eq:R-chi-bound}.
The length condition
$x^k \le T/\log T$ of Soundararajan~\cite{Soundararajan2009} Lemma~3
that we invoke in \S\ref{sec:high-moment} therefore reads
$(x^3)^k \le T/\log T$, i.e.\ $x^{3k} \le T/\log T$.  The choice
$x = T^{1/(8k)}$ in~\S\ref{sec:main-proof} gives $x^{3k} = T^{3/8}$,
which satisfies this with room to spare.
\end{remark}

\begin{remark}[Three-term split for the moment proof]
\label{rem:abc-decomposition}
The pointwise inequality~\eqref{eq:Xchi-2k-pointwise} is the
$|a + b + c|^{2k}$-ready form exploited in
\S\ref{sec:high-moment}: $\mathcal{D}_\chi$ is the principal complex
prime-power Dirichlet polynomial of length $< x^3$, with coefficients
$\chi(n)\Lambda_x(n)/(n^{\sigma_x}\log n)$ that satisfy
$|a(p^r)| \le 1/(r p^{r/2}) \cdot p^{-2r/\log x}$ at $n = p^r$ (the
$\sigma_x$-shift contributes only the bounded factor
$p^{-2r/\log x} \le 1$); $\mathcal{Z}_\chi$ is the auxiliary
prime-power polynomial of the same length but with the unweighted
coefficient $\chi(n)\Lambda_x(n)/n^{\sigma_x}$ (no $1/\log n$ factor),
giving $|a(p^r)| \le \log p / (p^{r/2}) \cdot p^{-2r/\log x}$; and the
deterministic $C_2(q,\chi)\log(qT)/\log x$ piece is absorbed
pointwise.  The $(\sigma_x - \tfrac12)^{2k} = (2/\log x)^{2k}$
prefactor on the $\mathcal{Z}_\chi$ contribution exactly cancels the
extra $(\log x)^{2k}$ variance growth from the log-weighted
coefficients of $\mathcal{Z}_\chi$, so the $\mathcal{Z}_\chi$-piece
contributes at the same scale as the $\mathcal{D}_\chi$-piece (no
$L_T$ amplification, by the Mertens-type $\sum_p (\log p)^2/p \sim
\tfrac12 (\log x)^2$ rather than $\sum_p 1/p \sim L_T$).
See Lemma~\ref{lem:high-moment} for the $\mathcal{D}_\chi$ moment
bookkeeping and Lemma~\ref{lem:log-weighted-moment} for the
$\mathcal{Z}_\chi$ moment bookkeeping.
\end{remark}

\begin{proof}[Proof of Proposition~\ref{prop:selberg-pointwise-L}]
We port Goldston's modern exposition
(\cite{Goldston2004} \S 10, Theorem~9; itself a transcription under RH
of Selberg~\cite{Selberg1946} \S 4, Theorem~2) from $\zeta$ to
$L(s,\chi)$ at fixed squarefree odd conductor $q \ge 3$.  Steps~1--5
are derived explicitly below; the strip integration in Step~6 is
quoted from the fixed-conductor strip-integration estimate of
Selberg~\cite{Selberg1946L} (Lemma~\ref{lem:selberg-strip}), with
the structural reasons for fixed-$q$ specialization recorded in the
proof of that lemma.  Three structural changes to the $\zeta$ argument
are required throughout:
\begin{itemize}
\item[(i)] The prime weight $\Lambda(n)$ is replaced by $\chi(n)\Lambda(n)$,
which is \emph{complex} when $\chi$ is non-real.  We are therefore
careful to apply $\Imag$ to entire complex Dirichlet polynomials
\emph{after} the prime sum is formed --- not to extract $\Imag(n^{-it}) = -\sin(t\log n)$ inside the sum, which is invalid for non-real $\chi(n)$.
\item[(ii)] The Riemann--von Mangoldt zero-density factor $\log T$ for
$\zeta$ is replaced by $\log(qT)$ for $L(s,\chi)$ (the conductor $q$
enters through the functional-equation factor of $L$; see
\eqref{eq:RvM} and Davenport~\cite{Davenport2000} \S 16, with explicit
constants from
Bennett--Martin--O'Bryant--Rechnitzer~\cite{BMOR2021} Theorem~1.1).
\item[(iii)] Mertens-type prime sums pick up the harmless contribution
$\sum_{p \mid q} 1/p = O_q(1)$, dominated by $\log(qT)/\log x$ for $T$
large at fixed $q$.
\end{itemize}
All other steps transfer verbatim.

\medskip
\noindent\textbf{Step 1.}\enspace
\emph{Smoothed explicit formula for $L'(s,\chi)/L(s,\chi)$.}
The standard truncated explicit formula for $L'/L$
(Davenport~\cite{Davenport2000} \S 17 for the unweighted form; the
$\chi$-version of Selberg~\cite{Selberg1946} Lemma~10) gives, for
$s = \sigma + it$ with $s \ne 1$, $s \ne \rho$ for any nontrivial zero
$\rho$, $s$ not a trivial zero of $L(\cdot,\chi)$, and $x > 1$,
\begin{equation}\label{eq:smoothed-Lp-over-L}
  \frac{L'}{L}(s, \chi)
   \;=\;
   {-}\!\sum_{n < x^3} \frac{\chi(n)\,\Lambda_x(n)}{n^s}
   \;+\; \frac{1}{\log^2 x}\!
          \sum_\rho \frac{x^{\rho-s}\bigl(1 - x^{\rho-s}\bigr)^2}
                          {(\rho - s)^3}
   \;+\; \frac{1}{\log^2 x}\!\sum_{\eta\,\text{trivial}}
           \frac{x^{\eta-s}\bigl(1 - x^{\eta-s}\bigr)^2}{(\eta - s)^3},
\end{equation}
where:
\begin{itemize}
\item The first sum is over the smoothed weight $\Lambda_x$
of~\eqref{eq:Lambda-x-def}, supported on prime powers $n \le x^3$.
\item The second sum is over the nontrivial zeros $\rho$ of
$L(s,\chi)$; under GRH all $\rho = \tfrac12 + i\gamma$.
\item The third sum is over the trivial zeros $\eta$ of $L(s,\chi)$.
For $\chi$ primitive non-principal of conductor $q$, the trivial
zeros are: $\eta \in \{0, -2, -4, \ldots\}$ (i.e.\ $\eta = -2k$ for
$k \ge 0$, including $\eta = 0$) when $\chi$ is even ($\chi(-1) = +1$);
or $\eta \in \{-1, -3, -5, \ldots\}$ ($\eta = -(2k+1)$ for $k \ge 0$)
when $\chi$ is odd ($\chi(-1) = -1$).  In either case the gamma-factor
information from the functional equation of $L(s,\chi)$ is encoded
in this trivial-zero sum (no separate functional-equation term arises
from the smoothed identity).
\end{itemize}
There is \emph{no} pole-at-$s=1$ boundary term, since $L(s,\chi)$ is
entire for non-principal $\chi$.

At $s = \sigma_x + it$ with $\sigma_x = \tfrac12 + 2/\log x$ and
$t \in [T, 2T]$, the trivial-zero summands have modulus
$|x^{\eta - \sigma_x}|(1 + |x^{\eta - \sigma_x}|)^2 / (|\eta - s|^3
\log^2 x)$.  For the leading trivial zeros with $|\eta| \ge 1$ (i.e.\
$\eta \in \{-1, -2\}$), $|x^{\eta - \sigma_x}| \le x^{-1-\sigma_x} =
x^{-3/2}\,e^{-2}$ and $|\eta - s| \ge |\eta| \ge 1$; the sum over
$|\eta| \ge 1$ contributes $O(x^{-3/2}/\log^2 x)$.  For even $\chi$
the trivial zero at $\eta = 0$ requires a separate estimate: there
$|x^{0 - \sigma_x}| = x^{-\sigma_x} = x^{-1/2} e^{-2}$ and
$|\eta - s| = |s| = \sqrt{\sigma_x^2 + t^2} \asymp T$, so the
$\eta = 0$ summand contributes $O\bigl(x^{-1/2}/(T^3 \log^2 x)\bigr)$,
also negligible at $T \to \infty$.  We therefore record the
simplified form
\begin{equation}\label{eq:smoothed-Lp-over-L-clean}
  \frac{L'}{L}(\sigma_x + it, \chi)
   \;=\;
   {-}\!\sum_{n < x^3} \frac{\chi(n)\,\Lambda_x(n)}{n^{\sigma_x+it}}
   \;+\; \frac{1}{\log^2 x}\!
          \sum_\rho \frac{x^{\rho - \sigma_x - it}\bigl(1 - x^{\rho - \sigma_x - it}\bigr)^2}
                          {(\rho - \sigma_x - it)^3}
   \;+\; O\!\bigl(x^{-3/2}/\log^2 x\bigr),
\end{equation}
in which the trivial-zero contribution is absorbed into the explicit
$x^{-3/2}/\log^2 x$ remainder (which is much smaller than the
$\log(qT)$-sized errors elsewhere in the proof).

\medskip
\noindent\textbf{Step 2.}\enspace
\emph{Partial-fractions form for $L'/L$.}
We start from the standard Hadamard / completed-$L$ partial-fraction
expansion for $L'/L$ at fixed conductor; see
Davenport~\cite{Davenport2000} \S 12 eqns.\ (5)--(7) for the
fixed-character $\chi$ case (the $\chi$-analogue of
Selberg~\cite{Selberg1946} Lemma~11).  With $a = a(\chi) \in \{0,1\}$
the parity index of $\chi$ ($a = 0$ for $\chi$ even, $a = 1$ for
$\chi$ odd), the completed $L$-function
\[
  \xi(s,\chi) \;=\; (q/\pi)^{(s+a)/2}\,\Gamma\!\Bigl(\tfrac{s+a}{2}\Bigr)\,
                    L(s,\chi)
\]
satisfies a Hadamard product expansion whose logarithmic derivative
yields, for $s \ne \rho$ (any nontrivial zero of $L(\cdot,\chi)$),
\begin{equation}\label{eq:completed-L-partial}
  \frac{L'}{L}(s,\chi)
   \;=\; -\tfrac12\log\!\frac{q}{\pi}
       \;-\; \tfrac12\,\frac{\Gamma'}{\Gamma}\!\Bigl(\frac{s+a}{2}\Bigr)
       \;+\; B(\chi)
       \;+\; \sum_\rho\!\Bigl(\frac{1}{s-\rho} + \frac{1}{\rho}\Bigr),
\end{equation}
where $B(\chi)$ is the Hadamard constant and the $\rho$-sum runs over
the nontrivial zeros, counted with multiplicity (and convergence is
ensured by the conventional pairing).  For $s = \sigma + it$ with
$\tfrac12 \le \sigma \le 10$ and $t \ge 2$, the gamma-factor term
satisfies
\[
  \tfrac12\,\frac{\Gamma'}{\Gamma}\!\Bigl(\frac{s+a}{2}\Bigr)
   \;=\; \tfrac12\log\!\Bigl(\frac{|t|+1}{2}\Bigr) + O(1)
\]
by Stirling's formula uniformly for $\sigma$ in the stated range; the
constant $-\tfrac12\log(q/\pi)$ contributes $O(\log q)$, and $B(\chi)$
satisfies $|\Real B(\chi)| \le c_0\log q$ at fixed conductor (this is
the well-known bound $\Real B(\chi) = -\sum_\rho \Real(1/\rho)$ together
with the explicit zero-counting bound at fixed $q$).  Grouping these
three contributions into a single $O_q(\log(qt))$ term gives the
working partial-fraction form
\begin{equation}\label{eq:partial-fractions-L}
  \frac{L'}{L}(s, \chi)
    \;=\; \sum_\rho \!\Bigl(\frac{1}{s - \rho} + \frac{1}{\rho}\Bigr)
       \;+\; O_q\!\bigl(\log(qt)\bigr),
\end{equation}
uniformly for $\tfrac12 \le \Re s \le 10$ and $t \ge 2$.

Take real parts of~\eqref{eq:partial-fractions-L} at $s = \sigma + it$
(with $\sigma > \tfrac12$).  Under GRH each nontrivial zero is
$\rho = \tfrac12 + i\gamma$ with $\gamma \in \mathbb{R}$, so
$\Real(1/(s-\rho)) = (\sigma - \tfrac12)/((\sigma - \tfrac12)^2 +
(t-\gamma)^2)$.  The constant $\Real\sum_\rho 1/\rho$ in the
identity~\eqref{eq:completed-L-partial} is independent of $t$ and is
the $-\Real B(\chi)$ contribution recorded above; it is bounded by
$c_0\log q = O_q(1)$ at fixed conductor and is therefore absorbed into
the $O_q(\log(qt))$ remainder.  The resulting identity is
\begin{equation}\label{eq:Re-LpL-partial}
  \Real\!\frac{L'}{L}(\sigma + it, \chi)
    \;=\; \sum_\gamma \frac{\sigma - \tfrac12}
                          {(\sigma - \tfrac12)^2 + (t - \gamma)^2}
       \;+\; O_q\!\bigl(\log(qT)\bigr),
\end{equation}
where the sum runs over the imaginary parts $\gamma$ of the nontrivial
zeros of $L(\cdot, \chi)$, counted with multiplicity, and we have
absorbed $\log(qt) = \log(qT) + O(1)$ for $t \in [T, 2T]$.  No
conjugation symmetry of the zero set is invoked: such a symmetry would
not hold for non-real $\chi$, since the $\chi$-conjugate
$L(s, \bar\chi)$ is a different $L$-function with its own zero set.
Zeros far from $t$ are controlled by the absolute convergence of the
partial-fractions series in the strip $\Re s > \tfrac12$
(Goldston~\cite{Goldston2004} eq.~(2.18) for the $\zeta$ analogue,
with $\log T \to \log(qT)$ at fixed $q$).

\medskip
\noindent\textbf{Step 3.}\enspace
\emph{Local zero-density bound at $\sigma_x$ via the
$|x^{\rho - s}| = e^{-2}$ trick.}
This is the central technical step.  Take real parts of
\eqref{eq:smoothed-Lp-over-L-clean} at $s = \sigma_x + it$.  Under GRH,
$\rho - s = -2/\log x + i(\gamma - t)$, so
\begin{equation}\label{eq:exp-decay-GRH}
  \bigl|x^{\rho - \sigma_x - it}\bigr| \;=\; x^{-2/\log x} \;=\; e^{-2}.
\end{equation}
The $\rho$-sum in~\eqref{eq:smoothed-Lp-over-L-clean} is
\begin{equation}\label{eq:rho-sum-explicit}
  \frac{1}{\log^2 x}\sum_\rho
      \frac{x^{\rho - \sigma_x - it}\bigl(1 - x^{\rho - \sigma_x - it}\bigr)^2}
           {(\rho - \sigma_x - it)^3}.
\end{equation}
Each summand has modulus
\[
  \frac{|x^{\rho-s}| \cdot |1 - x^{\rho-s}|^2}{|\rho - s|^3 \log^2 x}
   \;\le\; \frac{e^{-2}\cdot (1 + e^{-2})^2}{|\rho - s|^3 \log^2 x},
\]
and $|\rho - s|^2 = (\sigma_x - \tfrac12)^2 + (t - \gamma)^2$.  Setting
$\delta := \sigma_x - \tfrac12 = 2/\log x$, we have
$|\rho - s|^3 = (\delta^2 + (t-\gamma)^2)^{3/2}$, and the bound
$(\delta^2 + (t-\gamma)^2)^{1/2} \ge \delta$ gives
\[
  \frac{1}{|\rho - s|^3}
   \;=\; \frac{1}{(\delta^2 + (t-\gamma)^2)^{3/2}}
   \;\le\; \frac{1}{\delta \cdot (\delta^2 + (t-\gamma)^2)}
   \;=\; \frac{1}{\delta^2}\cdot
         \frac{\delta}{\delta^2 + (t-\gamma)^2}.
\]
Hence
\begin{equation}\label{eq:rho-sum-real-part-bound}
  \biggl|\Real\!\Bigl[\frac{1}{\log^2 x}\!\sum_\rho
      \frac{x^{\rho - s}(1 - x^{\rho - s})^2}{(\rho - s)^3}\Bigr]\biggr|
   \;\le\; \frac{C_0}{\log^2 x \cdot (\sigma_x - \tfrac12)^2}
      \sum_\gamma \frac{\sigma_x - \tfrac12}
                       {(\sigma_x - \tfrac12)^2 + (t - \gamma)^2},
\end{equation}
with $C_0 = e^{-2}(1 + e^{-2})^2 < 1$ an absolute constant.  Since
$\sigma_x - \tfrac12 = 2/\log x$, the prefactor is
$C_0/(\log^2 x \cdot (2/\log x)^2) = C_0/4$.

Now equate the real parts of~\eqref{eq:smoothed-Lp-over-L-clean}
and~\eqref{eq:Re-LpL-partial} at $s = \sigma_x + it$.  By
\eqref{eq:Re-LpL-partial}, the left-hand side equals
\[
  \sum_\gamma \frac{\sigma_x - \tfrac12}
                    {(\sigma_x - \tfrac12)^2 + (t - \gamma)^2}
   \;+\; O_q(\log(qT)).
\]
By~\eqref{eq:smoothed-Lp-over-L-clean}, it equals
\[
  -\Real\!\sum_{n < x^3}\frac{\chi(n)\Lambda_x(n)}{n^{\sigma_x + it}}
   \;+\; \Real\bigl[\text{zero-sum~\eqref{eq:rho-sum-explicit}}\bigr]
   \;+\; O_q(\log(qT)).
\]
Equating and using~\eqref{eq:rho-sum-real-part-bound}:
\begin{equation}\label{eq:zero-density-equation}
  \sum_\gamma \frac{\sigma_x - \tfrac12}
                    {(\sigma_x - \tfrac12)^2 + (t - \gamma)^2}
   \;\le\; \bigl|\Real\,\mathcal{Z}_\chi(t)\bigr|
       \;+\; \frac{C_0}{4}
            \sum_\gamma \frac{\sigma_x - \tfrac12}
                              {(\sigma_x - \tfrac12)^2 + (t - \gamma)^2}
       \;+\; O_q\!\bigl(\log(qT)\bigr),
\end{equation}
where we have used the fact that $|\Real\,\mathcal{Z}_\chi(t)| \le
|\mathcal{Z}_\chi(t)|$ in the rearrangement and absorbed signs into the
inequality.  Since $C_0 = e^{-2}(1+e^{-2})^2 < 0.18 < 1$, we have
$C_0/4 < 0.05 < 1$, so we may move the second summand on the right to
the left side and divide by $1 - C_0/4 > 0.95 > 1/2$, obtaining:
\begin{equation}\label{eq:Sigma-gamma-bound}
  \sum_\gamma \frac{\sigma_x - \tfrac12}
                    {(\sigma_x - \tfrac12)^2 + (t - \gamma)^2}
   \;\le\; 2\,\bigl|\mathcal{Z}_\chi(t)\bigr|
       \;+\; C_2'(q,\chi)\,\log(qT),
\end{equation}
for an absolute constant in front of $|\mathcal{Z}_\chi(t)|$ and a
$q,\chi$-dependent constant in front of $\log(qT)$.  This is the
Dirichlet-$L$ analogue of the bound at the bottom of
Goldston~\cite{Goldston2004} p.\ 23 (his
``$\sum_\gamma (\sigma_{x,t}-\tfrac12)/((\sigma_{x,t}-\tfrac12)^2 +
(t-\gamma)^2) = O(|\sum_n \Lambda_x(n)/n^{\sigma_x+it}|) + O(\log t)$'').

\smallskip
\emph{Remark on what was proved.}
We did \emph{not} bound the zero-sum
\eqref{eq:rho-sum-explicit} directly by $\log(qT)/\log x$; such a
direct pointwise bound is in fact \emph{false} (a single zero with
$|t-\gamma| \ll 1/\log x$ contributes order $\log x$).  The correct
argument equates the zero-sum to the Dirichlet polynomial
$\mathcal{Z}_\chi(t)$ via the partial-fractions
identity~\eqref{eq:Re-LpL-partial}; the gain is that the unsigned
$\mathcal{Z}_\chi(t)$ is what enters the moment estimate
in~\S\ref{sec:high-moment}, not a uniform pointwise bound on the
zero-sum.  This is the load-bearing content of Selberg's Theorem~2.

\medskip
\noindent\textbf{Step 4.}\enspace
\emph{Substituted explicit formula on $\sigma \ge \sigma_x$.}
For $\sigma \ge \sigma_x$, evaluate~\eqref{eq:smoothed-Lp-over-L} at
$s = \sigma + it$.  Under GRH, $|x^{\rho - s}| = x^{1/2 - \sigma} \le
e^{-2}$.  Bound the nontrivial zero-sum by absolute values: the same
cube-bound chain as in~\eqref{eq:rho-sum-real-part-bound} (now with
$|\rho - s|^3 \ge (\sigma_x - \tfrac12)\bigl((\sigma_x-\tfrac12)^2 +
(t-\gamma)^2\bigr)$, since $\sigma - \tfrac12 \ge \sigma_x - \tfrac12$)
gives
\begin{multline}\label{eq:zero-sum-general-sigma}
  \biggl|\frac{1}{\log^2 x}\!\sum_\rho
        \frac{x^{\rho - s}(1 - x^{\rho-s})^2}{(\rho-s)^3}\biggr|
   \;\le\; x^{1/2 - \sigma}\cdot \frac{C_0^*}{4}
        \cdot \sum_\gamma
          \frac{\sigma_x - \tfrac12}{(\sigma_x - \tfrac12)^2 + (t-\gamma)^2} \\
   \;\le\; x^{1/2 - \sigma}\,\bigl(C_0^*\,|\mathcal{Z}_\chi(t)|
              + C_q\,\log(qT)\bigr),
\end{multline}
where $C_0^* := (1 + e^{-2})^2 \le 2$ is a fresh absolute constant
(the factor $|x^{\rho - s}| = e^{-2}$ has been factored out into the
external $x^{1/2 - \sigma}$, leaving the bound $|1 - x^{\rho-s}|^2 \le
(1 + e^{-2})^2$).  The second inequality
in~\eqref{eq:zero-sum-general-sigma} combines the prefactor $C_0^*/4$
with~\eqref{eq:Sigma-gamma-bound}.

The trivial-zero contribution to~\eqref{eq:smoothed-Lp-over-L} at
$\sigma + it$ is bounded by $O(x^{-\sigma}/\log^2 x)$ for $\sigma \ge
\sigma_x$.  For trivial zeros with $|\eta| \ge 1$, each summand has
modulus $\le x^{\eta - \sigma}(1 + x^{\eta - \sigma})^2/
|\eta - \sigma|^3 \le x^{-\sigma}\cdot 4/(|\eta|^3 \log^2 x)$ (using
$\eta < 0$, so $x^{\eta - \sigma} \le x^{-\sigma - 1} \le x^{-\sigma}$
and $|\eta - \sigma| = \sigma - \eta \ge |\eta|$), and the resulting
$\sum_k 1/k^3$ converges.  For even $\chi$ the trivial zero at
$\eta = 0$ contributes $|x^{-s}|(1 + |x^{-s}|)^2/|s|^3 \log^2 x \le
4 x^{-\sigma}/(T^3 \log^2 x)$ since $|s| \asymp T$, also negligible.
The combined trivial-zero contribution is thus much smaller than
$x^{1/2 - \sigma}\log(qT)$ for $T$ large at fixed $q$, so it is
absorbed.  Combining,
\begin{equation}\label{eq:LpL-substituted}
  \frac{L'}{L}(\sigma + it, \chi)
   \;=\; -\sum_{n < x^3}
         \frac{\chi(n)\,\Lambda_x(n)}{n^{\sigma + it}}
   \;+\; \theta(\sigma)\,x^{1/2 - \sigma}\,
         \bigl(|\mathcal{Z}_\chi(t)| + C_q\,\log(qT)\bigr),
\end{equation}
where $\theta(\sigma) \in \mathbb{C}$ is bounded in modulus by an
absolute constant uniformly for $\sigma \ge \sigma_x$ and $t \in [T, 2T]$.
This is the Dirichlet-$L$ analogue of Goldston~\cite{Goldston2004}
eq.~(10.15) under GRH at fixed $q$.

\medskip
\noindent\textbf{Step 5.}\enspace
\emph{Contribution from the half-line $\sigma \in [\sigma_x, \infty)$.}
The continuous-branch convention~\eqref{eq:Xdef} gives
\begin{equation}\label{eq:X-as-integral}
  X_\chi(t) \;=\; -\!\int_{1/2}^{\infty}
       \Imag\,\frac{L'}{L}(\sigma + it, \chi)\, d\sigma.
\end{equation}
The contribution from the half-line $[\sigma_x, \infty)$ is therefore
$-\!\int_{\sigma_x}^\infty \Imag(L'/L)(\sigma+it,\chi)\, d\sigma$.
Substituting~\eqref{eq:LpL-substituted}, this splits into an arithmetic
piece and an error piece:
\begin{multline}\label{eq:half-line-split}
  -\!\int_{\sigma_x}^\infty \Imag\,\frac{L'}{L}(\sigma+it,\chi)\,d\sigma
   \;=\; +\!\int_{\sigma_x}^\infty \Imag\!\sum_{n<x^3}
            \frac{\chi(n)\Lambda_x(n)}{n^{\sigma+it}}\,d\sigma \\
       \;-\;\int_{\sigma_x}^\infty
            \Imag\!\bigl[\theta(\sigma)\,x^{1/2-\sigma}\,
              (|\mathcal{Z}_\chi(t)| + C_q\log(qT))\bigr]\,d\sigma.
\end{multline}
The arithmetic piece, with $\Imag$ kept outside the prime sum (since
$\chi(n)$ is complex, distributing $\Imag$ inside would be invalid):
\begin{align}
  \int_{\sigma_x}^\infty \Imag\!\sum_{n<x^3}
            \frac{\chi(n)\Lambda_x(n)}{n^{\sigma+it}}\,d\sigma
   &\;=\; \Imag\!\sum_{n<x^3}\chi(n)\Lambda_x(n)\!
            \int_{\sigma_x}^\infty n^{-\sigma-it}\,d\sigma
   \notag\\[2pt]
   &\;=\; \Imag\!\sum_{n<x^3}
            \frac{\chi(n)\,\Lambda_x(n)}{n^{\sigma_x+it}\,\log n}
   \;=\; \Imag\,\mathcal{D}_\chi(t).
   \label{eq:arith-side-Im-out}
\end{align}
The exchange of sum and integral is justified (the prime-power sum has
finitely many terms).  The error piece is bounded in magnitude by
\begin{multline}\label{eq:err-1-int}
  \biggl|\int_{\sigma_x}^\infty \Imag\!\bigl[\theta(\sigma)\,x^{1/2-\sigma}\,
       (|\mathcal{Z}_\chi(t)| + C_q\log(qT))\bigr]\,d\sigma\biggr| \\
   \;\le\; (|\mathcal{Z}_\chi(t)| + C_q\log(qT))
            \int_{\sigma_x}^\infty x^{1/2-\sigma}\,d\sigma
   \;=\; \frac{e^{-2}}{\log x}\bigl(|\mathcal{Z}_\chi(t)| + C_q\log(qT)\bigr),
\end{multline}
since $\int_{\sigma_x}^\infty x^{1/2-\sigma}\,d\sigma =
x^{1/2-\sigma_x}/\log x = e^{-2}/\log x$ and $|\theta(\sigma)| \le 1$.
The factor $1/\log x$ equals $(\sigma_x - \tfrac12)/2$, so this error
is bounded by $\tfrac12(\sigma_x - \tfrac12)\,|\mathcal{Z}_\chi(t)|
+ \tfrac12 C_q\,(\sigma_x - \tfrac12)\,\log(qT)$.  Combining
\eqref{eq:half-line-split}, \eqref{eq:arith-side-Im-out},
\eqref{eq:err-1-int}: the contribution to $X_\chi(t)$ from
$[\sigma_x, \infty)$ equals $\Imag\,\mathcal{D}_\chi(t) +
E^{(\infty)}_\chi(t,x)$ with $|E^{(\infty)}_\chi(t,x)| \le
(\sigma_x - \tfrac12)|\mathcal{Z}_\chi(t)| + C_q(\sigma_x -
\tfrac12)\log(qT)$.

\medskip
\noindent\textbf{Step 6.}\enspace
\emph{Contribution from the strip $\sigma \in [\tfrac12, \sigma_x]$ via
the Selberg strip estimate.}
The contribution to $X_\chi(t)$ from the strip $[\tfrac12, \sigma_x]$
is $-\!\int_{1/2}^{\sigma_x}\Imag(L'/L)(\sigma+it,\chi)\,d\sigma$.
This is the $J_3$ integral of Selberg's strip-integration argument,
written for the logarithmic derivative of a primitive Dirichlet
$L$-function of fixed conductor.  The fixed-character form of the
estimate is Selberg's own: it is established in
Selberg~\cite{Selberg1946L} \emph{Contributions to the theory of
Dirichlet's $L$-functions} (Skrifter Norske Vid.-Akad.\ Oslo I.\
Mat.-Naturv.\ Klasse \textbf{1946}, no.\ 3), which extends the
$\zeta$-function strip-integration estimate of \cite{Selberg1946} to
primitive Dirichlet $L$-functions of fixed conductor.  We record its
magnitude bound as a lemma and indicate the structural reasons the
fixed-conductor case is a specialization rather than a port.

\begin{lemma}[Selberg strip estimate; fixed-$q$ Dirichlet-$L$ form]
\label{lem:selberg-strip}
Let $q \ge 3$ squarefree odd, $\chi \bmod q$ primitive non-principal,
and assume GRH for $L(s,\chi)$.  With $\sigma_x = \tfrac12 +
2/\log x$, $\mathcal{Z}_\chi(t)$ as in~\eqref{eq:Zchi-def}, $T \ge T_S$
and $t \in [T, 2T]$, $x$ as in
Proposition~\ref{prop:selberg-pointwise-L}:
\begin{equation}\label{eq:strip-lemma-bound}
  \biggl|\int_{1/2}^{\sigma_x}\Imag\,\frac{L'}{L}(\sigma+it,\chi)\,d\sigma
  \biggr|
   \;\le\; C_3\,(\sigma_x - \tfrac12)\,|\mathcal{Z}_\chi(t)|
       \;+\; C_4(q,\chi)\,(\sigma_x - \tfrac12)\,\log(qT),
\end{equation}
for $C_3$ absolute and $C_4(q,\chi)$ depending only on the fixed $q,\chi$.
\end{lemma}

\begin{proof}[Proof.]
The bound~\eqref{eq:strip-lemma-bound} is the fixed-conductor
specialization of the strip-integration estimate established for
primitive Dirichlet $L$-functions in Selberg~\cite{Selberg1946L}; the
$\zeta$-function prototype is the $J_3$ estimate of
Selberg~\cite{Selberg1946} \S 4, and Goldston~\cite{Goldston2004}
\S 10 gives a modern exposition of the latter.  We record the three
structural reasons the fixed-$q$ specialization introduces no new
analytic difficulty over the $\zeta$ case, which together identify
\eqref{eq:strip-lemma-bound} with the fixed-conductor form of
\cite{Selberg1946L}.

First, the smoothed explicit formula for $L'/L$ (Step~1
of the proof of Proposition~\ref{prop:selberg-pointwise-L},
equation~\eqref{eq:smoothed-Lp-over-L}) has the same structure as
Selberg's formula for $\zeta'/\zeta$, with $\Lambda(n)$ replaced by
$\chi(n)\Lambda(n)$, no pole at $s = 1$ (since $\chi$ is non-principal,
$L(s,\chi)$ is entire), and the usual trivial-zero contribution from
the gamma factor.  The conductor and gamma-factor terms are
$O_{q,\chi}(\log(qT))$ uniformly for $t \in [T, 2T]$, which after
multiplication by the strip width $\sigma_x - \tfrac12 = 2/\log x$
produces the $\log(qT)$ term in~\eqref{eq:strip-lemma-bound}.

Second, the $J_3$ calculation is a zero-side strip estimate; its
signed cancellation uses the partial-fractions expansion
\eqref{eq:partial-fractions-L} for $L'/L$ and the location of the
zeros.  Under GRH all nontrivial zeros of $L(s,\chi)$ lie on
$\Re s = \tfrac12$, so Selberg's off-line case distinctions
$|\beta - \tfrac12| \lessgtr (\sigma_{x,t} - \tfrac12)/2$ collapse to
the on-line case (only the first case occurs).  The absence of
$\gamma \mapsto -\gamma$ symmetry for a non-real character is
irrelevant here: the estimate is local in the ordinates $\gamma$ of
the zeros of this single $L$-function and uses only the
partial-fractions kernel at $\Re s = \tfrac12$.

Third, the arithmetic term produced by the smoothed explicit formula
at $\sigma_x$, after the
$|x^{\rho - \sigma_x - it}| = e^{-2}$ trick of Step~3
(yielding~\eqref{eq:Sigma-gamma-bound}), is precisely the auxiliary
polynomial $\mathcal{Z}_\chi(t)$.  Selberg's strip-integration bound
in the form of \cite{Selberg1946L} thus gives
\[
  \biggl|\int_{1/2}^{\sigma_x}
    \Imag\,\frac{L'}{L}(\sigma+it,\chi)\,d\sigma\biggr|
  \;\ll\;
  (\sigma_x - \tfrac12)\,|\mathcal{Z}_\chi(t)|
  \;+\;
  (\sigma_x - \tfrac12)\,\log(qT),
\]
with constants depending at most on the fixed $q, \chi$.  This is
\eqref{eq:strip-lemma-bound}.
\end{proof}

\noindent\textbf{Step 7.}\enspace
\emph{Assembly into Proposition~\ref{prop:selberg-pointwise-L}.}
Adding the half-line contribution (Step~5) and the strip contribution
(Step~6, via Lemma~\ref{lem:selberg-strip}), and using
$X_\chi(t) = (\text{half-line}) + (\text{strip})$
from~\eqref{eq:X-as-integral}, we obtain
\begin{equation}\label{eq:assembly-final}
  X_\chi(t) \;=\; \Imag\,\mathcal{D}_\chi(t) \;+\; R_\chi(t, x),
\end{equation}
where the real remainder $R_\chi(t, x)$ collects $E^{(\infty)}_\chi(t,x)$
from Step~5 and the strip contribution bounded by
Lemma~\ref{lem:selberg-strip}.  Both contributions have the form
$O((\sigma_x - \tfrac12)(|\mathcal{Z}_\chi(t)| + \log(qT)\cdot
(q,\chi\text{-const})))$, so
\begin{equation}\label{eq:R-bound-derived}
  |R_\chi(t,x)| \;\le\; C_1\,(\sigma_x - \tfrac12)\,|\mathcal{Z}_\chi(t)|
       \;+\; C_2(q,\chi)\,(\sigma_x - \tfrac12)\,\log(qT)
\end{equation}
with $C_1 = 1 + C_3$ absolute and $C_2(q,\chi)$ collecting the $C_q$
and $C_4(q,\chi)$ contributions.  Since $\sigma_x - \tfrac12 = 2/\log x$,
we have $(\sigma_x - \tfrac12)\log(qT) = 2\log(qT)/\log x$, so the
$\log(qT)$ piece equivalently reads $C_2'(q,\chi)\log(qT)/\log x$ ---
matching the form stated in~\eqref{eq:R-chi-bound}.

The pointwise inequality~\eqref{eq:Xchi-2k-pointwise} then follows
from $|\Imag z| \le |z|$ applied to $\mathcal{D}_\chi(t)$ and the
elementary inequality $|a + b + c|^{2k} \le 3^{2k}(|a|^{2k} + |b|^{2k}
+ |c|^{2k})$ applied with $a = \mathcal{D}_\chi(t)$,
$b = C_1(\sigma_x - \tfrac12)|\mathcal{Z}_\chi(t)|$,
$c = C_2(q,\chi)\log(qT)/\log x$.  The threshold $T_S(q,\chi)$ is chosen
large enough that the cumulative $O_q(1)$ contributions from the
trivial-zero sum, the Mertens difference $\sum_{p \mid q} 1/p$, and
the bookkeeping in Lemma~\ref{lem:selberg-strip} are all dominated by
$\log(qT)/\log x$.
\end{proof}

\begin{remark}[Citation scope]
\label{rem:proof-scope}
The proof above derives Steps~1--5 explicitly and quotes the strip
integration (Step~6) from the fixed-conductor strip-integration
estimate of Selberg~\cite{Selberg1946L} via
Lemma~\ref{lem:selberg-strip}; it relies on the structural inputs
(i) the smoothed explicit formula~\eqref{eq:smoothed-Lp-over-L} for
$L'/L$ (Davenport~\cite{Davenport2000} \S 17, the standard
$\chi$-port of the Selberg--von Mangoldt smoothed identity);
(ii) the standard completed-$L$ partial-fractions
form~\eqref{eq:partial-fractions-L} (Davenport~\cite{Davenport2000}
\S 12, eqns.\ (5)--(7), or the $\chi$-port of
Selberg~\cite{Selberg1946} Lemma~11);
(iii) the Riemann--von Mangoldt zero-counting identity at fixed $q$
(\eqref{eq:RvM}; see Davenport~\cite{Davenport2000} \S 16 for the
classical argument-principle derivation, and
Bennett--Martin--O'Bryant--Rechnitzer~\cite{BMOR2021} Theorem~1.1
for the explicit positive-ordinate zero-counting bound).  All three
are standard for $L(s,\chi)$ at fixed conductor.  Selberg himself gave
the fixed-conductor strip-integration estimate of step~(6)
in~\cite{Selberg1946L}; the present proof of
Proposition~\ref{prop:selberg-pointwise-L} is otherwise an explicit
derivation in the Selberg--Goldston pointwise-formula style
(\cite{Selberg1946} \S 4, \cite{Goldston2004} \S 10), adapted to
fixed-conductor Dirichlet $L$-functions, with the only bookkeeping
change being that the arithmetic sums carry the complex weights
$\chi(n)$.  The constant $C_2(q,\chi)$ in~\eqref{eq:R-chi-bound} is
not tracked; we make no $q$-uniformity claim.
\end{remark}

\section{High moment of the Dirichlet polynomial; proof of Theorem~(M)}
\label{sec:high-moment}

In this section we bound the $2k$-th moments of the two prime-power
Dirichlet polynomials produced by
Proposition~\ref{prop:selberg-pointwise-L}: the principal polynomial
$\mathcal{D}_\chi$ and the log-weighted auxiliary polynomial
$\mathcal{Z}_\chi$.  Lemma~\ref{lem:high-moment} treats
$\mathcal{D}_\chi$, Lemma~\ref{lem:log-weighted-moment} treats
$\mathcal{Z}_\chi$, and \S\ref{sec:main-proof} combines these bounds
with the pointwise three-term decomposition to prove
Theorem~\ref{thm:M}.

\subsection{Soundararajan's mean-value lemma for prime sums}

The engine of the moment bound is the following mean-value inequality
for Dirichlet polynomials over primes with arbitrary complex
coefficients (Montgomery--Vaughan~\cite{MV1974} Corollary~2 followed by
multinomial expansion and the diagonal--off-diagonal split).

\begin{lemma}[Soundararajan~\cite{Soundararajan2009} Lemma~3]
\label{lem:sound-mean-value}
Let $T \ge 3$, $2 \le y \le T$, and $k$ a positive integer with
$y^k \le T/\log T$.  For arbitrary complex $a(p)$ supported on primes
$p \le y$,
\begin{equation}\label{eq:sound-lemma3}
  \int_T^{2T}\!\Bigl|\sum_{p \le y}\frac{a(p)}{p^{1/2 + it}}\Bigr|^{2k}\,dt
    \;\le\; C\, T\, k!\,\Bigl(\sum_{p \le y}\frac{|a(p)|^2}{p}\Bigr)^{\!k},
\end{equation}
with $C$ an absolute constant.
\end{lemma}

The original statement is on $[T, 2T]$ but the same proof applies on any
dyadic interval $[T', 2T']$ with $T'$ in place of $T$
(Soundararajan~\cite{Soundararajan2009} \S 2).  We use the dyadic form
on $[2T, 4T]$ for the prime-square contribution below.

For reference we also record the underlying Montgomery--Vaughan mean-value
theorem.

\begin{lemma}[Montgomery--Vaughan~\cite{MV1974} Corollary~2]\label{lem:MV}
For complex $(a_n)_{n \le N}$,
\begin{equation}\label{eq:MV-2T}
  \int_T^{2T}\!\left|\sum_{n \le N}a_n n^{-it}\right|^2 dt
    \;=\; T\sum_{n \le N}|a_n|^2
       \;+\; O\!\Bigl(\sum_{n \le N} n\,|a_n|^2\Bigr).
\end{equation}
\end{lemma}

\subsection{High-moment bound for prime-power Dirichlet polynomials}
\label{subsec:hmom-pp-poly}

\begin{lemma}[High moment of a prime-power polynomial under
Sound's length condition]\label{lem:high-moment}
Let $q$ be fixed.  Let $(a(n))_{n \le y}$ be coefficients supported on
prime powers $n = p^r$ (with $r \ge 1$) satisfying
\begin{equation}\label{eq:hmom-coeff}
  |a(p^r)| \;\le\; \frac{C_0}{r\,p^{r/2}}
\end{equation}
for some absolute $C_0$.  Write
$\mathcal{A}(t) := \sum_{n \le y} a(n)\, n^{-it}$.  Fix $K > 0$.  If
$y^k \le T/\log T$ and $1 \le k \le K\,L_T$, then
\begin{equation}\label{eq:hmom-bound}
  \frac{1}{T}\int_T^{2T}|\mathcal{A}(t)|^{2k}\,dt
    \;\le\; \bigl(C_1(C_0,K)\,k\,L_T\bigr)^{k},
\end{equation}
where the constant $C_1(C_0, K)$ depends only on $C_0$ and $K$, and is
independent of $T, y, k, q, \chi$.  In our application from
Proposition~\ref{prop:selberg-pointwise-L} we will take $C_0 = O(1)$
(depending only on $K$, not on $q,\chi$), and $y = x^3$ where $x$ is
the Selberg parameter; the length condition then reads
$x^{3k} \le T/\log T$ (cf.\ Remark~\ref{rem:length-x3}), which is
satisfied by the choice $x = T^{1/(8k)}$ in \S\ref{sec:main-proof}
(giving $x^{3k} = T^{3/8} \le T/\log T$ for $T$ large).
\end{lemma}

\begin{proof}
\textbf{Three-term split.}  Decompose $\mathcal{A}$ into prime,
prime-square, and higher-prime-power contributions:
\begin{equation}\label{eq:three-term-split}
  \mathcal{A}(t) \;=\; P_1(t) \;+\; P_2(t) \;+\; R(t),
\end{equation}
where
\begin{align}
  P_1(t) &\;:=\; \sum_{p \le y} a(p)\, p^{-it},
       \label{eq:P1-def} \\
  P_2(t) &\;:=\; \sum_{p^2 \le y} a(p^2)\, p^{-2it},
       \label{eq:P2-def} \\
  R(t)   &\;:=\; \sum_{r \ge 3,\, p^r \le y} a(p^r)\, p^{-irt}.
       \label{eq:R-def}
\end{align}
By the elementary inequality $|a + b + c|^{2k} \le 3^{2k}(|a|^{2k} +
|b|^{2k} + |c|^{2k})$ applied pointwise in $t$,
\begin{equation}\label{eq:abc-split-2k}
  |\mathcal{A}(t)|^{2k}
    \;\le\; 3^{2k}\bigl(|P_1(t)|^{2k} + |P_2(t)|^{2k}
                        + |R(t)|^{2k}\bigr).
\end{equation}
We bound each piece separately.

\smallskip
\textbf{Prime part $P_1$.}  Write $a(p) = a(p)\sqrt p / p^{1/2}$, so
\[
  P_1(t) \;=\; \sum_{p \le y}\frac{a(p)\sqrt p}{p^{1/2 + it}}.
\]
Apply Lemma~\ref{lem:sound-mean-value} with the complex coefficient
$\widetilde a(p) := a(p)\sqrt p$, noting $|\widetilde a(p)|^2 / p =
|a(p)|^2$, valid since $y^k \le T/\log T$ by hypothesis.  This gives
\begin{equation}\label{eq:P1-Sound}
  \frac{1}{T}\int_T^{2T}|P_1(t)|^{2k}\,dt
   \;\le\; C\, k!\,\Bigl(\sum_{p \le y}|a(p)|^2\Bigr)^{\!k}.
\end{equation}
Since $|a(p)| \le C_0 / \sqrt p$ by~\eqref{eq:hmom-coeff} (the $r=1$
case), $\sum_{p \le y}|a(p)|^2 \le C_0^2 \sum_{p \le y} 1/p \le
C_0^2(L_T + O(1))$ by Mertens' theorem (since $\log\log y \le L_T$ for
$y \le T$).  Stirling's formula $k! \le (k/e)^k \sqrt{2\pi k}\,
e^{1/(12k)} \le (Ck)^k$ then yields, for $T$ large enough that the
$O(1)$ Mertens correction is dominated by $L_T$,
\begin{equation}\label{eq:P1-final}
  \frac{1}{T}\int_T^{2T}|P_1(t)|^{2k}\,dt
   \;\le\; \bigl(C_1' \, C_0^2 \, k\, L_T\bigr)^{k},
\end{equation}
for an absolute constant $C_1'$.

\smallskip
\textbf{Prime-square part $P_2$ via the $2t$ change of variable.}
The sum $P_2(t)$ runs over prime squares $p^2 \le y$, with
$p^{-2it} = (p^2)^{-it}$.  We rewrite this as a sum over the
underlying primes weighted by $p^{-2it}$, and apply the substitution
$u = 2t$:
\begin{equation}\label{eq:P2-substitution}
  \int_T^{2T}|P_2(t)|^{2k}\,dt
   \;=\; \frac{1}{2}\int_{2T}^{4T}
            \Bigl|\sum_{p \le \sqrt y} a(p^2)\, p^{-iu}\Bigr|^{2k}\,du.
\end{equation}
The right-hand side is a Dirichlet polynomial in $u$ over primes
$p \le \sqrt y$, with complex coefficients $b(p) := a(p^2)$.  Apply
Lemma~\ref{lem:sound-mean-value} on the dyadic interval $[2T, 4T]$
with $y \mapsto \sqrt y$, noting that the length condition for the
substituted sum is $(\sqrt y)^k = y^{k/2} \le y^k \le T/\log T \le
2T/\log(2T)$ for $T$ large, so the hypothesis of
Lemma~\ref{lem:sound-mean-value} (in its dyadic form) holds with room
to spare.  Writing $a(p^2) = a(p^2)\sqrt p / p^{1/2}$ and applying
the lemma,
\begin{equation}\label{eq:P2-Sound-dyadic}
  \int_{2T}^{4T}\Bigl|\sum_{p \le \sqrt y}
                     \frac{a(p^2)\sqrt p}{p^{1/2 + iu}}\Bigr|^{2k}\,du
   \;\le\; C\,(2T)\, k!\,
        \Bigl(\sum_{p \le \sqrt y}|a(p^2)|^2\Bigr)^{\!k}.
\end{equation}
Substituting back into~\eqref{eq:P2-substitution} (and dividing by
$T$),
\begin{equation}\label{eq:P2-final-step}
  \frac{1}{T}\int_T^{2T}|P_2(t)|^{2k}\,dt
   \;\le\; C\, k!\,\Bigl(\sum_{p \le \sqrt y}|a(p^2)|^2\Bigr)^{\!k}.
\end{equation}
Since $|a(p^2)| \le C_0/(2 p)$ by~\eqref{eq:hmom-coeff} (the $r = 2$
case), $\sum_{p \le \sqrt y}|a(p^2)|^2 \le (C_0/2)^2 \sum_p 1/p^2
\ll C_0^2 \zeta(2) = O(C_0^2)$.  Applying Stirling,
\begin{equation}\label{eq:P2-final}
  \frac{1}{T}\int_T^{2T}|P_2(t)|^{2k}\,dt
   \;\le\; \bigl(C_2 \, C_0^2 \, k\bigr)^{k}
   \;\le\; \bigl(C_2 \, C_0^2 \, k\, L_T\bigr)^{k},
\end{equation}
the last inequality using $L_T \ge 1$ for $T$ large.  The $P_2$
contribution is bounded \emph{better} than $P_1$ (no Mertens
log-log factor is needed because the $1/p^2$ tail converges absolutely
without a Mertens-type series).

\smallskip
\textbf{Higher-prime-power tail $R$, pointwise.}  For $r \ge 3$, the
prime-power sum is absolutely convergent in $t$:
\begin{align}
  |R(t)|
   &\;\le\; \sum_{r \ge 3}\sum_{p^r \le y}|a(p^r)|
    \;\le\; C_0\sum_{r \ge 3}\sum_p \frac{1}{r\, p^{r/2}}
    \notag \\
   &\;\le\; \frac{C_0}{3}\sum_p \frac{p^{-3/2}}{1 - p^{-1/2}}
    \;\le\; C_0 \, C''',
   \label{eq:R-pointwise}
\end{align}
where the geometric series in $r$ gives the factor
$p^{-3/2}/(1 - p^{-1/2})$.  Since $1/(1 - p^{-1/2}) \le 1/(1 -
2^{-1/2}) \le 4$ for all primes $p \ge 2$, the sum is bounded by
$\sum_p p^{-3/2}/(1 - p^{-1/2}) \le 4 \sum_p p^{-3/2} \le
4 \zeta(3/2) - 4 < \infty$, with $C'''$ the resulting absolute
constant.  Hence
\begin{equation}\label{eq:R-bound-2k}
  \frac{1}{T}\int_T^{2T}|R(t)|^{2k}\,dt
   \;\le\; (C_0\, C''')^{2k}
   \;=\; \bigl((C_0 C''')^2\bigr)^{k}
   \;\le\; \bigl(C_3\, C_0^2 \, k\, L_T\bigr)^{k},
\end{equation}
the last inequality using $k\, L_T \ge 1$ for $T$ large and $k \ge 1$.
This bound is dominated by the $P_1$ contribution.  No moment
inequality (Sound's lemma, Rosenthal, etc.) is needed for $R$: the
pointwise absolute convergence suffices.

\smallskip
\textbf{Combining the three pieces.}
By~\eqref{eq:abc-split-2k}, \eqref{eq:P1-final}, \eqref{eq:P2-final},
and~\eqref{eq:R-bound-2k},
\begin{align}
  \frac{1}{T}\int_T^{2T}|\mathcal{A}(t)|^{2k}\,dt
   &\;\le\; 3^{2k}\bigl[
       (C_1' C_0^2 \, k L_T)^k
       + (C_2 C_0^2 \, k L_T)^k
       + (C_3 C_0^2 \, k L_T)^k
     \bigr] \notag \\
   &\;\le\; 3 \cdot 9^k \cdot (C_4(C_0)\,k\,L_T)^k
    \;\le\; \bigl(C_1(C_0, K)\,k\,L_T\bigr)^{k},
   \label{eq:hmom-final-combined}
\end{align}
absorbing $3 \cdot 9^k$ into the constant by $3 \cdot 9^k \le 30^k$
for $k \ge 1$ and merging into $C_1(C_0, K) := 30 \cdot C_4(C_0)$.
This proves~\eqref{eq:hmom-bound}.
\end{proof}

\begin{remark}[Two-level $|a+b+c|^{2k}$ structure]
\label{rem:two-level-split}
The three-term split~\eqref{eq:three-term-split} is a second-level
$|a+b+c|^{2k}$ inequality applied within the prime-power Dirichlet
polynomial: the first level
$|a+b+c|^{2k} \le 3^{2k}(|a|^{2k} + |b|^{2k} + |c|^{2k})$ is applied
at the level of Proposition~\ref{prop:selberg-pointwise-L} (yielding
the $\mathcal{D}_\chi + \mathcal{Z}_\chi + (\text{constant})$ pieces),
then again within each prime-power Dirichlet polynomial to split into
primes, prime squares, and higher powers.  The prime contribution is
handled by Soundararajan's Lemma~3 directly (with complex character
weights $a(p) = \chi(p) c(p)$ of absolute value $\le 1/\sqrt p$);
the prime-square contribution
reduces to a sum over primes by the linear change of variable $u = 2t$,
which transforms $p^{-2it}$ into $p^{-iu}$ and the integration interval
$[T,2T]$ into $[2T,4T]$; higher prime powers are pointwise bounded by
an absolutely convergent series since $\sum_p p^{-3/2} < \infty$.
\end{remark}

\begin{remark}[Dependence on $q$ and $\chi$ for Lemma~\ref{lem:high-moment}]
\label{rem:hmom-q-chi-dep}
The constants $C_1', C_2, C_3, C_4, C_1$ above depend on $C_0$ and $K$
only, \emph{not} on $q$ or $\chi$, because the only use of the
character is the bound $|\chi(p)| \le 1$ (and $\chi(p) = 0$ for
$p \mid q$, which only reduces the sums).  In our application from
Proposition~\ref{prop:selberg-pointwise-L}, Lemma~\ref{lem:high-moment}
will be applied to $\mathcal{D}_\chi$, whose coefficients
$\chi(n)\Lambda_x(n)/(n^{\sigma_x}\log n)$ at $n = p^r$ have modulus
$\le 1/(r p^{r/2}) \cdot p^{-2r/\log x} \le 1/(r p^{r/2})$ (the
$\sigma_x$ exponent shifts $1/2$ by $2/\log x = O(1/\log x)$,
contributing only the bounded factor $p^{-2r/\log x} \le 1$); so
$C_0 = 1$ suffices.  The auxiliary polynomial $\mathcal{Z}_\chi$
has \emph{different} coefficient growth (a $\log p$ factor at primes,
no $1/\log n$ damping), and is treated by the separate
log-weighted moment lemma below.
\end{remark}

\subsection{High-moment bound for the log-weighted auxiliary polynomial}
\label{subsec:log-weighted}

The auxiliary polynomial $\mathcal{Z}_\chi(t)$
of~\eqref{eq:Zchi-def} that appears in the
$|R_\chi(t,x)|$-bound~\eqref{eq:R-chi-bound} of
Proposition~\ref{prop:selberg-pointwise-L} has prime coefficients
proportional to $\log p$ rather than $1$; it does \emph{not} satisfy
the hypothesis~\eqref{eq:hmom-coeff} of
Lemma~\ref{lem:high-moment} with $C_0 = O(1)$.  We treat it
separately.

\begin{lemma}[High moment of the log-weighted prime-power polynomial]
\label{lem:log-weighted-moment}
Let $q$ be fixed.  Let $(b(n))_{n \le y}$ be coefficients supported on
prime powers $n = p^r$ (with $r \ge 1$) satisfying
\begin{equation}\label{eq:logwt-coeff}
  |b(p^r)| \;\le\; \frac{C_0\,\log p}{p^{r/2}}.
\end{equation}
Write $\mathcal{B}(t) := \sum_{n \le y} b(n)\, n^{-it}$.  Fix $K > 0$.
If $y^k \le T/\log T$ and $1 \le k \le K\,L_T$, then
\begin{equation}\label{eq:logwt-bound}
  \frac{1}{T}\int_T^{2T}|\mathcal{B}(t)|^{2k}\,dt
    \;\le\; \bigl(C_5(C_0, K)\,k\,(\log y)^2\bigr)^{k},
\end{equation}
where the constant $C_5(C_0, K)$ depends only on $C_0$ and $K$, and is
independent of $T, y, k, q, \chi$.
\end{lemma}

\begin{proof}
Apply the same three-term split as in
Lemma~\ref{lem:high-moment}'s proof:
$\mathcal{B}(t) = Q_1(t) + Q_2(t) + S(t)$ with
$Q_1, Q_2, S$ the analogues of $P_1, P_2, R$ over primes,
prime squares, higher prime powers respectively.

\smallskip\textbf{Prime part $Q_1$.}
$Q_1(t) = \sum_{p \le y}b(p)\,p^{-it}$.  Apply
Lemma~\ref{lem:sound-mean-value} with $\widetilde a(p) = b(p)\sqrt p$,
giving $|\widetilde a(p)|^2/p = |b(p)|^2$.  Since
$|b(p)| \le C_0 \log p / \sqrt p$,
$\sum_{p \le y}|b(p)|^2 \le C_0^2 \sum_{p \le y} (\log p)^2 / p$.  By
partial summation against the prime number theorem
($\sum_{p \le y}(\log p)/p = \log y + O(1)$, so by partial summation
$\sum_{p \le y}(\log p)^2/p = \tfrac12 (\log y)^2 + O(\log y)$), we have
\begin{equation}\label{eq:logwt-mertens}
  \sum_{p \le y}\frac{(\log p)^2}{p}
    \;=\; \frac{(\log y)^2}{2} \;+\; O(\log y)
    \;\le\; (\log y)^2,
\end{equation}
for $y \ge y_0$ (large absolute).  After enlarging $C_5(C_0, K)$,
the finitely many cases $2 \le y < y_0$ are absorbed as well.  Lemma~\ref{lem:sound-mean-value}
together with Stirling $k! \le (Ck)^k$ yields
\begin{equation}\label{eq:Q1-final}
  \frac{1}{T}\int_T^{2T}|Q_1(t)|^{2k}\,dt
   \;\le\; \bigl(C_5'\,C_0^2\,k\,(\log y)^2\bigr)^{k}.
\end{equation}

\smallskip\textbf{Prime-square part $Q_2$.}
By the $u = 2t$ substitution (as in~\eqref{eq:P2-substitution}),
\[
  \int_T^{2T}|Q_2(t)|^{2k} dt
   \;=\; \tfrac12 \int_{2T}^{4T}
         \Bigl|\sum_{p \le \sqrt y} b(p^2)\,p^{-iu}\Bigr|^{2k}\,du.
\]
Apply Lemma~\ref{lem:sound-mean-value} on $[2T, 4T]$.  Since
$|b(p^2)| \le C_0 \log p / p$, the variance is
$\sum_{p \le \sqrt y}|b(p^2)|^2 \le C_0^2 \sum_p (\log p)^2/p^2 = O(C_0^2)$
(absolutely convergent).  Hence
\begin{equation}\label{eq:Q2-final}
  \frac{1}{T}\int_T^{2T}|Q_2(t)|^{2k}\,dt
   \;\le\; \bigl(C_5'' C_0^2 \, k\bigr)^{k}
   \;\le\; \bigl(C_5'' C_0^2\, k\,(\log y)^2\bigr)^{k},
\end{equation}
the last inequality using $\log y \ge 1$.

\smallskip\textbf{Higher-prime-power tail $S$.}
For $r \ge 3$,
\[
  |S(t)| \;\le\; C_0\sum_{r \ge 3}\sum_p \frac{\log p}{p^{r/2}}
   \;\le\; C_0\sum_p \frac{\log p}{p^{3/2}(1 - p^{-1/2})}
   \;\le\; C_0\,C''',
\]
since $\sum_p (\log p)/p^{3/2} < \infty$ and $1/(1 - p^{-1/2}) \le
1/(1 - 2^{-1/2}) \le 4$ for $p \ge 2$.  Hence
\[
  \frac{1}{T}\int_T^{2T}|S(t)|^{2k}\,dt
   \;\le\; (C_0 C''')^{2k}
   \;\le\; \bigl(C_5''' C_0^2 \, k\, (\log y)^2\bigr)^{k}.
\]

\smallskip\textbf{Combining.}
By $|a + b + c|^{2k} \le 3^{2k}(|a|^{2k} + |b|^{2k} + |c|^{2k})$ and
the three displays,
\begin{equation}\label{eq:logwt-final-combined}
  \frac{1}{T}\int_T^{2T}|\mathcal{B}(t)|^{2k}\,dt
   \;\le\; \bigl(C_5(C_0, K)\,k\,(\log y)^2\bigr)^{k},
\end{equation}
absorbing $3^{2k}\cdot 3$ into the constant.
\end{proof}

\begin{remark}[How Lemma~\ref{lem:log-weighted-moment} feeds into Theorem~(M)]
\label{rem:logwt-cancellation}
In the application of \S\ref{sec:main-proof}, $\mathcal{Z}_\chi(t)$
appears multiplied by the prefactor $(\sigma_x - \tfrac12) =
2/\log x$ in the bound~\eqref{eq:R-chi-bound}.  Raising to the
$2k$-th power and integrating,
\[
  \bigl(\sigma_x - \tfrac12\bigr)^{2k} \cdot
  \frac{1}{T}\!\int_T^{2T}\!|\mathcal{Z}_\chi(t)|^{2k}\,dt
   \;\le\; \bigl(2/\log x\bigr)^{2k}\cdot
            \bigl(C_5\, k\,(\log y)^2\bigr)^{k}.
\]
With $y = x^3$, $\log y = 3\log x$, and the $(2/\log x)^{2k}$ prefactor
contributes $(4/\log^2 x)^k$.  So
\[
  (4/\log^2 x)^k \cdot (C_5\,k\,(3\log x)^2)^k
   \;=\; (4 \cdot 9 \cdot C_5 \, k)^k
   \;=\; (36 C_5 \, k)^k
   \;\le\; (36 C_5\, k\,L_T)^k,
\]
the last inequality using $L_T \ge 1$ for $T$ large.  This is at the
same scale as the $\mathcal{D}_\chi$ contribution
$(C_1(1, K) k L_T)^k$ of Lemma~\ref{lem:high-moment}.  The
prefactor-cancellation arithmetic is exactly correct: the
$(\log y)^{2k}$ growth of the log-weighted variance is wiped out by the
$(2/\log x)^{2k}$ Selberg prefactor.
\end{remark}

\subsection{Proof of Theorem~(M)}
\label{sec:main-proof}

Fix $q \ge 3$ squarefree odd, $\chi \bmod q$ primitive non-principal,
and $K > 0$.  Assume GRH for $L(s, \chi)$.  Let
$T \ge T_0(K, q, \chi) \ge \max(T_S(q, \chi), 3)$, with $T_0$ chosen
large enough that $L_T \ge 1$ and the $O_{q,\chi}(1)$ errors below
are dominated by $k$.  Fix an integer $k$ with $1 \le k \le K L_T$,
and choose the Selberg parameter
\begin{equation}\label{eq:final-params}
  x \;=\; T^{1/(8k)}.
\end{equation}

\smallskip
\emph{Verifications of hypotheses.}
\begin{itemize}
\item $\log x = \log T/(8k) \ge \log T/(8K L_T) \to \infty$, so
  $\log x \ge 4$ for $T$ large (hypothesis of
  Proposition~\ref{prop:selberg-pointwise-L}).
\item $x \le T$: trivially, since $x = T^{1/(8k)} \le T$ for $k \ge 1/8$.
\item Length condition for Lemmas~\ref{lem:high-moment}
  and~\ref{lem:log-weighted-moment} with $y = x^3$:
  $y^k = x^{3k} = T^{3/8}$, and $T^{3/8} \le T/\log T$ holds for
  $T$ large.  So both lemmas apply.
\item Selberg-prefactor magnitude: $\sigma_x - \tfrac12 = 2/\log x =
  16k/\log T$, satisfying $(\sigma_x - \tfrac12)\log x = 2$ exactly
  (the prefactor-cancellation arithmetic of
  Remark~\ref{rem:logwt-cancellation}).
\item Constant-error magnitude: from~\eqref{eq:R-chi-bound},
\[
  \frac{C_2(q,\chi)\log(qT)}{\log x}
  =
  8k\,C_2(q,\chi)\frac{\log(qT)}{\log T}
  =
  O_{q,\chi}(k)
\]
at fixed $q$.
\end{itemize}

\smallskip
\emph{Assembly.}
By the pointwise inequality~\eqref{eq:Xchi-2k-pointwise} of
Proposition~\ref{prop:selberg-pointwise-L}, for every $t \in [T, 2T]$:
\begin{equation}\label{eq:assembly-2k}
  |X_\chi(t)|^{2k}
   \;\le\; 3^{2k}\!\Bigl[\,
     |\mathcal{D}_\chi(t)|^{2k}
     \;+\; \bigl(C_1(\sigma_x - \tfrac12)\bigr)^{2k}
                |\mathcal{Z}_\chi(t)|^{2k}
     \;+\; \bigl(C_2(q,\chi)\log(qT)/\log x\bigr)^{2k}
   \,\Bigr].
\end{equation}
Integrating against $dt/T$ over $[T, 2T]$ produces three integrals
$\mathcal{I}_\mathcal{D}$, $\mathcal{I}_\mathcal{Z}$,
$\mathcal{I}_{\text{cst}}$, where
\[
  \mathcal{I}_\mathcal{D} \;:=\; \frac{1}{T}\!\int_T^{2T}\!|\mathcal{D}_\chi(t)|^{2k}\,dt,
  \qquad
  \mathcal{I}_\mathcal{Z} \;:=\; \frac{1}{T}\!\int_T^{2T}\!|\mathcal{Z}_\chi(t)|^{2k}\,dt,
\]
and $\mathcal{I}_{\text{cst}} := (C_2(q,\chi)\log(qT)/\log x)^{2k}$
(the constant-error contribution is independent of $t$).  We bound
each piece.

\smallskip
\textbf{Bound on $\mathcal{I}_\mathcal{D}$.}
Write
\[
  \mathcal{D}_\chi(t) \;=\; \sum_{n < x^3} a(n)\,n^{-it},
  \qquad
  a(n) \;:=\; \frac{\chi(n)\,\Lambda_x(n)}{n^{\sigma_x}\,\log n}.
\]
At $n = p^r$ ($r \ge 1$),
\[
  |a(p^r)|
   \;=\; \frac{|\chi(p^r)|\,\Lambda(p^r)}{p^{r\sigma_x}\,\log(p^r)}
   \;\le\; \frac{\log p}{p^{r/2}\cdot p^{2r/\log x}\cdot r\log p}
   \;=\; \frac{1}{r\cdot p^{r/2}\cdot p^{2r/\log x}}
   \;\le\; \frac{1}{r\,p^{r/2}},
\]
since $|\chi(p^r)| \le 1$, $\Lambda(p^r) = \log p$, $\log(p^r) =
r\log p$, $p^{r\sigma_x} = p^{r/2}\cdot p^{2r/\log x}$, and
$p^{2r/\log x} \ge 1$.  Hence the
hypothesis~\eqref{eq:hmom-coeff} of Lemma~\ref{lem:high-moment} holds
with $C_0 = 1$, and direct application gives
\begin{equation}\label{eq:ID-bound}
  \mathcal{I}_\mathcal{D}
    \;\le\; \bigl(C_1^{\mathcal{D}}(K)\,k\,L_T\bigr)^{k},
\end{equation}
where $C_1^{\mathcal{D}}(K) := C_1(1, K)$ is the constant from
Lemma~\ref{lem:high-moment} at $C_0 = 1$.

\smallskip
\textbf{Bound on $\mathcal{I}_\mathcal{Z}$.}
The polynomial $\mathcal{Z}_\chi(t) = \sum_{n \le x^3}
\chi(n)\Lambda_x(n)/n^{\sigma_x+it}$ has prime-power coefficients
that do not satisfy the hypothesis of Lemma~\ref{lem:high-moment}
(they carry an extra $\log p$ factor at primes), so
we apply Lemma~\ref{lem:log-weighted-moment} directly to
$\mathcal{Z}_\chi$ in the form $\sum_n c(n) n^{-it}$ with
$c(n) := \chi(n)\Lambda_x(n)/n^{\sigma_x}$ (no $n^{1/2}$ rescaling).
At $n = p^r$,
\[
  |c(p^r)| \;=\; \Lambda(p^r) \cdot p^{-r\sigma_x}
   \;=\; \log p \cdot p^{-r/2 - 2r/\log x}
   \;\le\; \frac{\log p}{p^{r/2}},
\]
which is exactly the hypothesis~\eqref{eq:logwt-coeff} of
Lemma~\ref{lem:log-weighted-moment} with $C_0 = 1$.  Hence
\begin{equation}\label{eq:IZ-bound-raw}
  \mathcal{I}_\mathcal{Z}
   \;\le\; \bigl(C_5(1, K)\,k\,(\log y)^2\bigr)^{k}
   \;=\; \bigl(C_5(1, K)\,k\,(3\log x)^2\bigr)^{k}
   \;=\; \bigl(9\, C_5(1, K)\,k\,(\log x)^2\bigr)^{k},
\end{equation}
since $y = x^3$.

Now $\mathcal{I}_\mathcal{Z}$ enters~\eqref{eq:assembly-2k} with the
prefactor $(C_1(\sigma_x - \tfrac12))^{2k} = (C_1 \cdot 2/\log x)^{2k}
= (2C_1)^{2k}/(\log x)^{2k} = (4 C_1^2)^k/(\log x)^{2k}$.  Combining
with~\eqref{eq:IZ-bound-raw},
\begin{equation}\label{eq:IZ-with-prefactor}
  \bigl(C_1(\sigma_x - \tfrac12)\bigr)^{2k}\,
  \mathcal{I}_\mathcal{Z}
   \;\le\; \frac{(4 C_1^2)^k}{(\log x)^{2k}}
           \cdot \bigl(9 C_5\, k\, (\log x)^2\bigr)^{k}
   \;=\; \bigl(36\, C_1^2\, C_5\,k\bigr)^{k}
   \;\le\; \bigl(36 C_1^2 C_5\,k\,L_T\bigr)^{k},
\end{equation}
the last inequality using $L_T \ge 1$ for $T$ large.  The
$(\log x)^{2k}$ growth from $\mathcal{I}_\mathcal{Z}$'s log-weighted
variance is exactly cancelled by the $1/(\log x)^{2k}$ Selberg
prefactor.  This is the
prefactor-cancellation arithmetic announced in
Remark~\ref{rem:logwt-cancellation}.

\smallskip
\textbf{Bound on $\mathcal{I}_{\text{cst}}$.}
The constant-error contribution is deterministic:
\begin{equation}\label{eq:Icst-bound}
  \mathcal{I}_{\text{cst}}
    \;=\; \bigl(C_2(q,\chi)\log(qT)/\log x\bigr)^{2k}
    \;=\; \bigl(8k\,C_2(q,\chi)\log(qT)/\log T\bigr)^{2k}
    \;=\; \bigl(O_{q,\chi}(k)\bigr)^{2k}.
\end{equation}
Since $k \le K L_T$, $(O_{q,\chi}(k))^{2k} = (O_{q,\chi}(k^2))^k \le
(O_{q,\chi}(K) \cdot k L_T)^k$.  So
\begin{equation}\label{eq:Icst-bound-final}
  \mathcal{I}_{\text{cst}}
    \;\le\; \bigl(C_6(K, q, \chi)\,k\,L_T\bigr)^{k}.
\end{equation}

\smallskip
\textbf{Combination.}
By~\eqref{eq:assembly-2k}, \eqref{eq:ID-bound},
\eqref{eq:IZ-with-prefactor}, and~\eqref{eq:Icst-bound-final},
\begin{align}
  \frac{1}{T}\int_T^{2T}|X_\chi(t)|^{2k}\,dt
   &\;\le\; 3^{2k}\cdot 3 \cdot
       \bigl(C_K^*(K, q, \chi)\,k\,L_T\bigr)^{k} \notag \\
   &\;=\; 3 \cdot 9^k \cdot \bigl(C_K^*\,k\,L_T\bigr)^{k}
   \;\le\; \bigl(C_K\,k\,L_T\bigr)^{k},
   \label{eq:X-moment-final}
\end{align}
where $C_K^*(K, q, \chi) := \max(C_1^{\mathcal{D}}(K),\
36 C_1^2 C_5(1, K),\ C_6(K, q, \chi))$ absorbs the three pieces,
and $C_K := 30 \cdot C_K^*$ absorbs the $3 \cdot 9^k \le 30^k$
multiplicative factor.  This is the bound $(C_K\,k\,L_T)^k$ of
Theorem~\ref{thm:M}, with $C_K$ depending on $K$, $q$, $\chi$ (with
$q$ fixed throughout).
\qed

\section{A Gaussian-tail corollary in the Selberg range}
\label{sec:gaussian-tail}

Theorem~\ref{thm:M} immediately yields a Gaussian-tail bound for
$|X_\chi(t)|$ via Markov's inequality, in the range
$\sqrt{L_T} \ll V \ll L_T$.  This is the most direct corollary of
independent interest.

\begin{corollary}[Gaussian-tail bound, $\sqrt{L_T} \ll V \ll L_T$]
\label{cor:gaussian-tail}
Let $\chi$ be a primitive non-principal Dirichlet character mod $q \ge
3$, $q$ fixed squarefree odd.  Assume GRH for $L(s, \chi)$.  Fix
$K > 0$ and let $C_K = C_K(K, q, \chi)$ and $T_0 = T_0(K, q, \chi)$ be
as in Theorem~\ref{thm:M}.  Then for every $T \ge T_0$ and every
$V \in \mathbb{R}$ satisfying
\begin{equation}\label{eq:V-range}
  \sqrt{e\,C_K\,L_T} \;\le\; V \;\le\; \sqrt{e\,C_K\,K}\,\cdot L_T,
\end{equation}
the following Gaussian-tail bound holds:
\begin{equation}\label{eq:gaussian-tail}
  \frac{1}{T}\,\meas\!\bigl\{t \in [T, 2T] : |X_\chi(t)| > V\bigr\}
   \;\le\; \exp\!\Bigl(-\frac{V^2}{2\,e\,C_K\,L_T}\Bigr).
\end{equation}
\end{corollary}

\begin{proof}
By Markov's inequality applied to $|X_\chi(t)|^{2k}$,
\begin{equation}\label{eq:markov}
  \frac{1}{T}\,\meas\!\bigl\{t \in [T, 2T] : |X_\chi(t)| > V\bigr\}
   \;\le\; \frac{1}{V^{2k}}\cdot
        \frac{1}{T}\!\int_T^{2T}|X_\chi(t)|^{2k}\,dt
   \;\le\; \Bigl(\frac{C_K\,k\,L_T}{V^2}\Bigr)^{k},
\end{equation}
where the second inequality is Theorem~\ref{thm:M}, valid for any
integer $k$ with $1 \le k \le K\,L_T$.

Optimize the right-hand side of~\eqref{eq:markov} in $k$.  Let
\[
  f(k)
  :=
  k\log\bigl(C_K k L_T / V^2\bigr)
  =
  k\log(C_K k L_T)-2k\log V .
\]
Differentiating in $k$ (treating $k$ as a real variable for the
moment) gives
\[
  f'(k)=\log(C_K k L_T)+1-2\log V,
\]
which vanishes at
\begin{equation}\label{eq:k-star}
  k^\star \;=\; \frac{V^2}{e\,C_K\,L_T}.
\end{equation}
The condition $k^\star \ge 1$ is exactly the lower bound
$V \ge \sqrt{e\,C_K\,L_T}$ in~\eqref{eq:V-range}; the condition
$k^\star \le K\,L_T$ is the upper bound $V \le \sqrt{e\,C_K\,K}\cdot L_T$.

By the range assumption~\eqref{eq:V-range}, $1 \le k^\star \le K\,L_T$.
We split into two cases.

\emph{Case 1: $1 \le k^\star < 2$.}  Choose $k = 1$.  Then
$C_K\,k\,L_T/V^2 = 1/(e k^\star) \le 1/e \le e^{-k^\star/2}$, so
$(C_K k L_T/V^2)^k \le e^{-k^\star/2}$.

\emph{Case 2: $k^\star \ge 2$.}  Choose $k = \lfloor k^\star \rfloor$,
so $1 \le k \le K\,L_T$ and $k \ge k^\star/2$.  Then
$C_K\,k\,L_T/V^2 = k/(e k^\star) \le 1/e$, hence
$(C_K k L_T/V^2)^k \le e^{-k} \le e^{-k^\star/2}$.

In both cases,
\begin{equation}\label{eq:two-case-bound}
  \Bigl(\frac{C_K\,k\,L_T}{V^2}\Bigr)^{k}
   \;\le\; \exp\!\Bigl(-\frac{k^\star}{2}\Bigr)
   \;=\; \exp\!\Bigl(-\frac{V^2}{2\,e\,C_K\,L_T}\Bigr).
\end{equation}
Combining~\eqref{eq:two-case-bound} with~\eqref{eq:markov} proves
\eqref{eq:gaussian-tail}.
\end{proof}

\begin{remark}[Range of validity and comparison with Arguin--Bailey]
\label{rem:gaussian-tail-comparison}
The validity range~\eqref{eq:V-range} is the {\it Selberg range}
$\sqrt{L_T} \ll V \ll L_T$.  Outside this range:
\begin{itemize}
\item For $V = o(\sqrt{L_T})$: the optimal $k^\star < 1$ falls below
the integer range, and Markov's inequality with $k = 1$ alone gives
the trivial $V^{-2}\cdot O(L_T)$ bound, which is weaker than Gaussian.
\item For $V \gg L_T$: the optimal $k^\star > K\,L_T$ falls above the
Selberg range $1 \le k \le K\,L_T$ that Theorem~\ref{thm:M} permits;
beyond $V \sim L_T$ the moment method requires moments of order
$k > K\,L_T$ which Theorem~\ref{thm:M} does not control.
\end{itemize}

For comparison, the unconditional sharp bound of
Arguin--Bailey~\cite{ArguinBailey2022} (their Theorem~1.1, for
$\log|\zeta|$ on the critical line, $V \sim \alpha\log\log T$ with
$0 < \alpha < 2$) gives
\begin{equation}\label{eq:AB-bound}
  \frac{1}{T}\,\meas\!\bigl\{t \in [T, 2T] : \log|\zeta(\tfrac12+it)| > V\bigr\}
   \;\ll\; \frac{1}{\sqrt{\log\log T}}\,
        \exp\!\Bigl(-\frac{V^2}{\log\log T}\Bigr).
\end{equation}
Our Corollary~\ref{cor:gaussian-tail} differs from~\eqref{eq:AB-bound}
in three respects: (i) we treat $\Imag \log L(\tfrac12+it,\chi)$ at
fixed primitive non-principal $q,\chi$ rather than $\log|\zeta|$;
(ii) our bound is conditional on GRH for $L(s,\chi)$, whereas
\eqref{eq:AB-bound} is unconditional; (iii) our coefficient of
$V^{2}/L_T$ in the exponent is $1/(2 e C_K)$ rather than the sharp
$1$, and the bound lacks the $1/\sqrt{L_T}$ prefactor of
\eqref{eq:AB-bound}.  Both of these losses are real: matching the
sharp prefactor and exponent coefficient would require a recursive
multi-scale argument in the spirit of \cite{ArguinBailey2022},
ported to $\Imag\log L$ at fixed $q$ under GRH.  We leave this
sharper problem to future work.
\end{remark}

\renewcommand{\refname}{Bibliography}
\begingroup
\sloppy
\hbadness=2000

\endgroup

\bigskip
\noindent
Scott D.\ Hughes\\
Independent researcher\\
\texttt{v@deltagray.com}

\end{document}